 \documentclass[final,3p]{elsarticle}

\usepackage{hyperref}
\hypersetup{colorlinks=true,linkcolor=black,citecolor=black}

\usepackage{amssymb,amsmath,amsthm}
\usepackage{supertabular,multirow,booktabs}

\usepackage{amsmath}
\usepackage{graphics,graphicx,epsf,subfigure,epstopdf}
\usepackage{amsfonts}
\usepackage{amssymb}
\usepackage{float}
\usepackage{color}
\usepackage{algorithm}
\usepackage{algorithmic}
\usepackage{savesym}
\usepackage{datetime}
\usepackage[normalem]{ulem}
\usepackage{supertabular,multirow,booktabs}
\usepackage{framed} 
\usepackage{xcolor}
\usepackage[utf8]{inputenc}
\usepackage{siunitx}

\colorlet{shadecolor}{gray!20}

\def\e{{\epsilon}}

\definecolor{darkmagenta}{rgb}{0.62, 0.0, 0.77}
\definecolor{darkgreen}{rgb}{0.13, 0.55, 0.13}
\definecolor{orange}{rgb}{0.98, 0.3, 0.0}
\definecolor{darkorange}{rgb}{1.0, 0.27, 0.0}
\definecolor{caribbeangreen}{rgb}{0.0, 0.8, 0.6}
\definecolor{brown}{rgb}{0.55, 0.27, 0.07}
\definecolor{green2}{rgb}{0.5, 0.7, 0.2}

\newcommand{\ts}[1]{\textcolor{darkmagenta}{ #1}}

\newcommand{\R}{\mathbb{R}}
\newcommand{\Rn}{\mathbb{R}^{n}}
\newcommand{\mR}{\mathcal{R}}

\def\im{\mathrm{i}}
\def\e{\mathrm{e}}
\def\Rbkk{{\mathbb{R}^{2k\times 2k}}}
\def\Rbnk{{\mathbb{R}^{2n\times 2k}}}

\def\Spkn{{\mathrm{Sp}(2k,2n)}}
\def\Spn{{\mathrm{Sp}(2n)}}

\newcommand{\skewset}{{\cal S}_{\mathrm{skew}}}
\newcommand{\symset}{{\cal S}_{\mathrm{sym}}}
\newcommand{\TU}{{\mathrm{T}_{X}}\Spkn}

\newcommand{\proj}{\mathcal{P}_X^{}}
\newcommand{\projn}{\mathcal{P}_X^\perp}
\newcommand{\proje}{\mathcal{P}_{X,e}^{}}
\newcommand{\projen}{\mathcal{P}_{X,e}^{\perp}}
\newcommand{\projc}{\mathcal{P}_{X,c}^{}}
\newcommand{\projcn}{\mathcal{P}_{X,c}^{\perp}}

\DeclareMathOperator*{\diag}{diag}

\DeclareMathOperator*{\qgeo}{qgeo}

\DeclareMathOperator*{\range}{im}

\DeclareMathOperator*{\skewsym}{skew}

\DeclareMathOperator*{\tr}{tr}
\DeclareMathOperator*{\grad}{grad}
\DeclareMathOperator*{\real}{Re}
\DeclareMathOperator*{\imag}{Im}

\newcommand{\abs}[1]{\left|#1\right|}
\newcommand{\dkh}[1]{\left(#1\right)}
\newcommand{\fkh}[1]{\left[#1\right]}

\newcommand{\rgrade}[1]{\mathrm{grad}_e f(#1)}
\newcommand{\rgradc}[1]{\mathrm{grad}_c f(#1)}					

\usepackage{tikz,array}
\usetikzlibrary{shapes.geometric}
\usetikzlibrary{shapes.arrows}
\usetikzlibrary{calc}
\usetikzlibrary{intersections}

\newtheorem{theorem}{Theorem}[section]
\newtheorem{proposition}[theorem]{Proposition}
\newdefinition{definition}[theorem]{Definition}
\newtheorem{lemma}[theorem]{Lemma}

\newdefinition{remark}[theorem]{Remark}

\begin{document}
\date{}
\begin{frontmatter}



\title{Optimization on the symplectic Stiefel manifold:\\
	SR decomposition-based retraction and applications}



\author[1]{Bin Gao}

\author[2]{Nguyen Thanh Son\corref{mycorrespondingauthor}}
\cortext[mycorrespondingauthor]{Corresponding author}
\ead{ntson@tnus.edu.vn}

\author[3]{Tatjana Stykel}

\address[1]{State Key Laboratory of Scientific and Engineering Computing, Academy of Mathematics and Systems Science, Chinese Academy of Sciences, 100190 Beijing, China}
\address[2]{Institut f\"ur Mathematik $\,\&$ Centre for Advanced Analytics and Predictive Sciences,  Universit\"at Augsburg, Universit\"{a}tsstra\ss e 12a, 86159 Augsburg, Germany \\
	and Thai Nguyen University of Sciences, 24118 Thai Nguyen, Vietnam }
\address[3]{Institut f\"ur Mathematik $\,\&$ Centre for Advanced Analytics and Predictive Sciences, Universit\"at Augsburg, Universit\"atsstra\ss e 12a, 86159 Augsburg, Germany} 


\begin{abstract}
Numerous problems in optics, quantum physics, stability analysis, and control of dynamical systems can be brought to an optimization problem with matrix variable subjected to the symplecticity constraint. As this constraint nicely forms a so-called symplectic Stiefel manifold, Riemannian optimization is preferred, because one can borrow ideas from unconstrained optimization methods after preparing necessary geometric tools. 
Retraction is arguably the most important one which decides the way iterates are updated given a~search direction. Two retractions have been constructed so far: one relies on the Cayley transform and the other is designed using quasi-geodesic curves.  In this paper, we propose a~new retraction which is based on an SR matrix  decomposition. We prove that its domain contains the open unit ball which is essential in proving the global convergence of the associated gradient-based optimization algorithm. Moreover, we consider three applications---symplectic target matrix problem, symplectic eigenvalue computation, and symplectic model reduction of Hamiltonian systems---with various examples. The extensive numerical comparisons reveal the strengths of the proposed optimization algorithm.
\end{abstract}

%


\end{frontmatter}



\section{Introduction}\label{Sec:Intro}
In scientific computing and physics, one often 
has to work with structured matrices. One type of 
them is the set of symplectic matrices defined~as
\begin{equation*}
\Spkn := \{ X \in \mathbb{R} ^{2n\times 2k}\quad : \quad X^TJ_{2n}X=J_{2k}\},
\end{equation*}
where $
J_{2n}=\left[\begin{smallmatrix}
0&I_n\\-I_n&0
\end{smallmatrix}\right]
$
is the $2n\times 2n$ skew-symmetric Poisson matrix and $I_n$ denotes the $n\times n$ identity matrix. This set has been indicated to be a closed, unbounded, embedded submanifold of $\Rbnk$, termed as the \emph{symplectic Stiefel manifold}~\cite{GSAS21}. 
For $k=n$, this manifold forms the symplectic Lie group denoted by $\Spn$. 
Another view on the symplectic Stiefel manifold is to consider it as a quotient manifold of the two Lie groups $\Spn/\mathrm{Sp}(2n-2k)$, see \cite{BendZ21} for details.
In this paper, we consider a~minimization problem with the symplecticity constraint given by
\begin{equation}
\min_{X\in \Spkn} f(X)
\label{eq:min}
\end{equation}
with a continuously differentiable cost function $f$.
In order to solve such an equality-constrained problem, one can apply, for example, a~standard penalty method or an~augmented Lagrangian approach \cite[Chapter~17]{NoceW06}. 
However, by exploiting the rich structure of $\Spkn$, Riemannian optimization becomes preferred as in this setting, 
various well-known unconstrained optimization methods in Euclidean spaces can be extended to the case of nonlinear manifold. 

To pursue this direction, a~Riemannian structure of the symplectic Stiefel manifold $\Spkn$ and necessary geometric tools are required. The tangent spaces of $\Spkn$ have been completely characterized in  \cite{GSAS21}. Moreover, a~canonical-like metric and an~Euclidean metric have been introduced in~\cite{GSAS21} and  \cite{GSAS21a}, respectively,  based on which the normal spaces,  the orthogonal projections onto the tangent and normal spaces and the Riemannian gradients have been investigated. In order to perform a search on $\Spkn$, we need a~mapping that retracts the Riemannian gradient from the tangent space at the current point to the manifold. To this end, two retractions have been proposed: the first one is based on the Cayley transform, henceforth referred to as the \emph{Cayley retraction}, and the other is constructed to mimic the shape of a~geodesic, called the \emph{quasi-geodesic retraction}. Based on these retractions, a~non-monotone line search algorithm has been developed in \cite{GSAS21} and various tests have been performed there to validate it. An~equivalent expression for the Cayley retraction on $\Spkn$ has also been derived  in \cite{BendZ21} starting from a~pseudo-Riemannian exponential.

In this paper, we propose a~new retraction which is based on 
an~SR~matrix decomposition \cite{DDora75,BunG86,Sala05}. 
Such a~retraction is therefore referred to as the \emph{SR~retraction}. We also establish in Proposition~\ref{prop:SGS} and Theorem~\ref{theo:SRretraction} that for rectangular $2n\times 2k$ matrices, the SR~decomposition almost always 
exists. More importantly, we show in Theorem~\ref{theo:SRretraction_existence} that the update can be computed if the length of the update with respect to the spectral matrix norm is smaller than one. Like the existing retractions, the new one can be combined with any metric which results in new Riemannian optimization schemes whose global convergence can be shown in a way similar to \cite[Theorem~5.7]{GSAS21}.

The minimization problem \eqref{eq:min} appears in various areas of physics and scientific computing. For example, the task of averaging optical transference matrices, investigating beam dynamics, and optimal control of quantum gate can be formulated as the minimization problem \eqref{eq:min} with $k=n$, see, e.g., \cite{Draftetal88,Harr04,WuCR08}. Further, stability analysis of  weakly damped gyroscopic systems~\cite{Lanc13} can be performed by solving \eqref{eq:min} with the trace cost function which enables determining so-called symplectic eigenvalues \cite{BhatJ15,SonAGS21,SonSt22}.
This fact also motivates us to perform an~extensive comparison of the gradient-based optimization schemes on a~wide variety of problems from different applications. These consist of the symplectic target problem which 
arises in optimal control of symplectic quantum gate, the computation of symplectic eigenvalues of a~symmetric positive-(semi)definite matrix which can be used for stability analysis of gyroscopic systems, and  model reduction of Hamiltonian systems. 

Especially, we would like to emphasize the application to structure-pre\-ser\-ving model reduction of Hamiltonian systems. Such systems have a specific structure and they possess some underlying physical properties such as conservation of energy, described by the Hamiltonian function, and conservation of mass which should be preserved during the model reduction process. This can be achieved by using the proper symplectic decomposition (PSD) model reduction approach developed in \cite{PengM16}. It consists in determining a~symplectic reduced basis matrix from a set of snapshots, i.e., solutions at different time instances,  which minimizes the symplectic projection error in a least squares sense. This leads to a~nonlinear optimization problem of the form~\eqref{eq:min}.
Unlike the orthogonal case, where a~singular value decomposition (SVD) of the data matrix can be used, an~explicit solution to the problem with the simplecticity constraint is still unknown. To simplify this problem, in~\cite{PengM16}, the optimal candidates are limited to the ones that additionally have orthonormal columns. The same condition is also used in a~greedy approach for parametric Hamiltonian systems in~\cite{AfkhH17}. Alternatively, the additional orthonormality condition  is eliminated  in~\cite{BuchBH19}, but the approach there is based on the \emph{SVD-like decomposition} and therefore, as those in~\cite{PengM16}, its solution is in general not guaranteed to be optimal  although it can have this property for a~special class of Hamiltonian systems in a~restrictive setting~\cite{BuchGH22}. 

In contrast to these approaches, in this paper, thanks to the progress on optimization techniques on the symplectic Stiefel manifold, we address the symplectic model reduction problem in an~optimal way. A~similar approach was recently considered in~\cite{BendZ22} in the framework of the symplectic Grassmann manifold~\cite{BendZ21} but no model reduction errors were reported there.
For nonlinear Hamiltonian systems,  the symplectic projection is combined with an~approximation of the nonlinear term computed by the discrete empirical interpolation decomposition method (DEIM) proposed in \cite{ChaS10} and its structure-preserving variant~\cite{ChatBG16}. In our experiments, we numerically evaluate the optimization-based model reduction methods and compare them with other structure-preserving reduction techniques deve\-loped for Hamiltonian systems.

After presenting the general notation, we organize the paper as follows. In 
Section~\ref{sec:Stiefel}, we briefly review basic geometric concepts and facts for the symplectic Riemannian manifold $\Spkn$. Notably, we present in detail the canonical-like and Euclidean metrics, and the corresponding formulations of the Riemannian gradient of the cost function. Section~\ref{sec:retraction} is devoted to  retractions,  which are indispensable in Riemannian optimization. We start with reviewing  the Cayley and quasi-geodesic retractions. Then, in the main part of this section, we introduce a new retraction based on an~SR decomposition and discuss the existence conditions and computational issues. Section~\ref{sec:linesearch} recalls the non-monotone line search algorithm based on the Riemannian gradient for solving the minimization problem \eqref{eq:min}. In Section~\ref{sec:appl}, we present three applications with several test models and report  on numerical results illustrating the properties of different Riemannian optimization schemes. Finally, the concluding remarks are given in Section~\ref{sec:concl}.

\textbf{Notation.} 
We denote by $\symset(m)$ and $\skewset(m)$ the sets of all $m\times m$ real symmetric and skew-symmetric  matrices, respectively, and $\skewsym(A)=\tfrac{1}{2}(A-A^T)$ stands for the skew-symmetric part of a square matrix $A$. The determinant, the trace, and the image of a~matrix $A$ are denoted by $\det(A)$, $\tr(A)$, and $\range(A)$, respectively. If $A$ is a complex matrix, $\real(A)$ and $\imag(A)$ denote the real and imaginary part of~$A$, respectively. For $A_{1},\ldots,A_{\ell}$ being square matrices, we denote by $\diag(A_{1},\ldots,A_{\ell})$ the block diagonal matrix.  The Frobenius and spectral matrix norms are denoted by $\|\cdot\|_F$ and $\|\cdot\|_2$, respectively, and the Euclidean vector norm is denoted by $\|\cdot\|$. For a~matrix $M\in\Rbkk$ and $j=1,\ldots,k$, we denote  by $(M)_{1:2j,1:2j}$ the $2j\times 2j$ leading principal submatrix of $M$. 
The dimension of a~subspace $\,\mathcal{U}\subset\mathbb{R}^{n\times m}$ is denoted
by $\dim(\mathcal{U})$.
Finally, for a function~$h$ defined on the Euclidean space $\Rbnk$, $\nabla h$ denotes the standard Euclidean gradient of~$h$.

\section{Riemannian geometry of the symplectic Stiefel manifold}
\label{sec:Stiefel}

In this section, we briefly review the geometric structure of the symplectic Stiefel manifold $\Spkn$ studied recently 
in \cite{GSAS21,GSAS21a}. Let us start with the result confirming that  $\Spkn$ is a~smooth embedded submanifold of the Euclidean space $\mathbb{R}^{2n\times 2k}$ and it has  dimension $4nk-k(2k-1)$, see  \cite[Proposition~3.1]{GSAS21}. Alternatively, one can show that $\Spkn$ is diffeomorphic to a~quotient space $\mathrm{Sp}(2n)/\mathrm{Sp}(2n-2k)$ and, hence, it admits a~structure of a~quotient manifold \cite[Proposition 3.1]{BendZ21}. Further, it has been shown in \cite[Proposition~3.3]{GSAS21} that the \emph{tangent space} of $\Spkn$ at $X\in \Spkn$ is given by 
\begin{align}\label{eq:tangsp_Euclidean1}
\TU &= \!\left\{Z \in \Rbnk\;:\; Z^TJ_{2n}X + X^TJ_{2n}Z = 0\right\}\\
&=\!\left\{ XJ_{2k}W \!+ \!J_{2n}X_\perp K \;:\; W\!\in\!\mathcal{S}_{\rm sym}(2k),
K\!\in\!\mathbb{R}^{(2n-2k)\times 2k} \right\}\!,
\label{eq:tangsp_Euclidean}
\end{align}
where $X_\perp\in\mathbb{R}^{2n\times(2n-2k)}$ has full rank and satisfy $X^TX_\perp=0$.
The choice for $X_\perp$ is obviously not unique and has a certain effect on numerical performance. It has been shown numerically that $X_\perp$ with orthonormal columns, i.e.,  $X_\perp^TX^{}_\perp=I$, is preferred, see \cite[Section 6]{GSAS21} for details. In what follows, we will restrict ourselves to this choice. 

The Riemannian structure of the symplectic Stiefel manifold $\Spkn$ greatly depends on the metric. Let $g_{X}:\TU\times \TU\to\mathbb{R}$ denote a~Riemannian metric on $\Spkn$ at $X$. Note that sometimes, the employed metrics does not depend on $X$. To simplify the notation, we omit the subscript. The \emph{normal space} to $\Spkn$ at $X\in\Spkn$ with respect to $g$ is defined as 
\[
\bigl(\TU\bigr)^{\perp} = \bigl\{ N\in\mathbb{R}^{2n\times 2k}\; : \; g(N,Z)=0 \text{ for all } Z\in \TU\bigr\}.
\]
It is well known that any $Y\in\Rbnk$ can be decomposed as
$$
Y=\proj(Y)+\projn(Y),
$$ 
where
$\proj$ and $\projn$ denote the orthogonal projections with respect to $g$ onto the tangent and normal spaces, respectively. 

The \emph{Riemannian gradient} of a~differentiable function $f:\Spkn\to\mathbb{R}$ at $X\in\Spkn$ with respect to the metric $g$, denoted by $\grad f(X)$, is defined as the unique element of $\TU$ that satisfies the condition 
\[
g\bigl(\grad f(X),Z\bigr) = \mathrm{D}\bar{f}(X)[Z]\qquad
\text{ for all } Z\in \TU,
\] 
where $\bar{f}$ is a~smooth extension of $f$ around $X$ in $\mathbb{R}^{2n\times 2k}$, and $\mathrm{D}\bar{f}(X)$ denotes the Fr\'echet derivative of $\bar{f}$ at~$X$. Using 
\cite[(3.37)]{AbsiMS04}, the Riemannian gradient can be determined as 
\begin{equation}\label{eq:grad_g}
\grad f(X)=\proj\bigl(\grad \bar{f}(X)\bigr).
\end{equation}

Next, we introduce the canonical-like and Euclidean metrics considered in \cite{GSAS21,GSAS21a} and present particular geometric concepts related to these metrics.

\subsection{Canonical-like metric}
A class of metrics on $\Spkn$ has been introduced in \cite{GSAS21}. For a~parameter $\rho>0$ and the tangent vectors $Z_i=XJ_{2k}W_i+J_{2n}X_\perp K_i$ with $W_i\in\symset(2k)$ and $K_i\in\R^{(2n-2k)\times 2k}$ for $i=1,2$, the \emph{canonical-like metric} is defined as
\begin{align*}
g_{c,\rho}(Z_1,Z_2) &:= \frac{1}{\rho}  \tr(W_1^T W_2^{})+\tr(K_1^T K_2^{}).
\end{align*}
The normal space with respect to this metric is then given by 
\[
\bigl(\TU\bigr)_{c}^{\perp} = \bigl\{ XJ_{2k}\varOmega  \enskip: \enskip \varOmega\in\skewset(2k)\bigr\}.
\]
Further, the corresponding orthogonal projections $\projc$ and $\projcn$ onto the tangent and normal spaces have the following form
\[
\projc(Y) =  S_{X,Y}J_{2n}X, \qquad 
\projcn(Y) =  XJ_{2k}\skewsym(X^TJ_{2n}^TY),
\]
where $Y\in\Rbnk$, and
\[
S_{X,Y} =G_XY(XJ_{2k})^T+XJ_{2k}(G_XY)^T, \quad
G_X =I_{2n}-\frac{1}{2}XJ_{2k}^{}X^TJ_{2n}^T,
\]
see \cite[Proposition~4.3]{GSAS21}.
Using \eqref{eq:grad_g}, the Riemannian gradient of a~function $f$ with respect to the metric $g_{c,\rho}$ can then be represented as
\[
\rgradc{X}=\projc(\mathrm{grad}_c\bar{f}(X)) =S_{{X,\nabla \bar{f}}} J_{2n}X
\]
with 
\[
{S_{X,\nabla \bar{f}}}  =H_X\nabla \bar{f}(X)(XJ_{2k})^T+XJ_{2k}(H_X\nabla \bar{f}(X))^T,\quad
H_X  =\frac{\rho}{2}XX^T+J_{2n}X_{\perp}^{}X_{\perp}^TJ_{2n}^T,
\]
see \cite[Proposition~4.5]{GSAS21} for details. 

\subsection{Euclidean metric}
Another metric on $\Spkn$ has been investigated in~\cite{GSAS21a}. For $Z_i=XJ_{2k}W_i+J_{2n}X_\perp K_i\in \TU$ with the matrices $W_i\in\symset(2k)$ and $K_i\in\R^{(2n-2k)\times 2k}$ for $i=1,2$, the 
\emph{Euclidean metric} is defined as
\begin{align*}
g_e(Z_1,Z_2) &:= \tr(Z_1^T Z_2) \\
&~= \tr(W_1^T J_{2k}^T X^T X J_{2k}^{} W_2^{})+
\tr(K_1^T X^T_\perp X_\perp^{} K_2^{})  \\
&\quad\enskip + \tr(W_1^T J_{2k}^T X^T J_{2n}{} X_\perp^{} K_2^{}) + 
\tr(K_1^T X^T_\perp J_{2n}^T X J_{2k}^{} W_2^{}).
\end{align*}
Based on this metric, the normal space to $\Spkn$ can be represented as 
\begin{equation}\label{eq:normal-Eucliean}
\bigl(\TU\bigr)_e^{\perp} = \bigl\{J_{2n}X\varOmega \enskip :\enskip \varOmega\in\skewset(2k)\bigr\},
\end{equation}
see \cite[Proposition~1]{GSAS21a}. Further, using \eqref{eq:tangsp_Euclidean} and \eqref{eq:normal-Eucliean}, we obtain the following expressions for the orthogonal projections $\proje$ and $\projen$ 
onto the tangent  and normal spaces, respectively, with respect to the Euclidean metric \cite[Proposition~2]{GSAS21a}:
\[
\proje(Y) =  Y-J_{2n}X{\varOmega_{X,Y}} ,\qquad
\projen(Y) =  J_{2n}X{\varOmega_{X,Y}}, 
\]
where $\varOmega_{X,Y}\in\skewset(2k)$ is the solution of the Lyapunov equation 
\begin{equation}\label{eq:lyapunov}
X^T X\varOmega+\varOmega\, X^T X =2 \skewsym(X^T J_{2n}^T Y).
\end{equation}
The existence and uniqueness of the solution of the Lyapunov equation \eqref{eq:lyapunov} immediately follows from the fact that the coefficient matrix $X^TX$ is symmetric and positive definite~\cite{GoluV13}. 

Since $\Spkn$ is endowed with the Euclidean metric $g_e$, the corresponding Riemannian gradient $\rgrade{X}$ can directly be calculated by using \eqref{eq:grad_g} as follows 
\begin{equation*} 
\rgrade{X}  = \proje(\mathrm{grad}_e\bar{f}(X)) = \proje(\nabla \bar{f}(X)) = \nabla \bar{f}(X)-J_{2n}X\varOmega_{X,\nabla \bar{f}},
\end{equation*}
where $\varOmega_{X,\nabla \bar{f}}\in\skewset(2k)$ solves the Lyapunov equation 
\begin{equation*}
X^T X\,\varOmega+\varOmega\, X^T X =2 \skewsym\dkh{X^T J_{2n}^T \nabla \bar{f}(X)}.
\end{equation*}

We summarize geometrical notions and their formulations for different metrics on the symplectic Stiefel manifold~$\Spkn$ in Table~\ref{tab:notion_summary}.

\begin{table}[htbp]
	\centering
	\small
	\caption{Geometric objects on the symplectic Stiefel manifold $\Spkn$ for different metrics. 
		Here, $Z_i=XJ_{2k}W_i+J_{2n}X_\perp K_i$ with $W_i\in\symset(2k)$ and $K_i\in\R^{(2n-2k)\times 2k}$, and the definitions of $S_{X,Y}$, $S_{X,\nabla \bar{f}}$, $\varOmega_{X,Y}$, and $\varOmega_{X,\nabla \bar {f}}$ can be found in Section~\textup{\ref{sec:Stiefel}}.}
	\label{tab:notion_summary}
	\begin{tabular}{llll}
		\toprule
		\multicolumn{2}{c}{} & Canonical-like & Euclidean \\\cmidrule[.6pt](r){3-3}\cmidrule[.6pt](r){4-4}
		{metric} & {$g(Z_1,Z_2)$} & {$\frac{1}{\rho}  \tr(W_1^T W_2^{})+\tr(K_1^T K_2^{})$~~}  & {$\tr(Z_1^T Z_2)$}
		\\\cmidrule(r){3-4}
		{normal space}~~ & $\bigl(\TU\bigr)^{\perp}$~~ &  {${XJ_{2k}\varOmega  : \varOmega\in\skewset(2k)}$} & {${J_{2n}X\varOmega  : \varOmega\in\skewset(2k)}$} 
		\\\cmidrule(r){3-4}
		\multirow{2}{*}{projection} & {$\proj(Y)$}  & $S_{X,Y}J_{2n}X$  & $Y-J_{2n}X{\varOmega_{X,Y}}$  \\\cmidrule[.1pt](r){3-4}
		& $\projn(Y)$ & $XJ_{2k}\skewsym(X^TJ_{2n}^TY)$ & $J_{2n}X{\varOmega_{X,Y}}$\\\cmidrule(r){3-4}
		{gradient} & $\grad f(X)$ &  $S_{X,\nabla\bar{f}}\,J_{2n}X$ & $\nabla \bar{f}(X)-J_{2n}X\varOmega_{X,\nabla\bar{f}}$  \\
		\bottomrule
	\end{tabular}
\end{table}

\section{Retractions on $\Spkn$}
\label{sec:retraction}

Retractions play a key role in Riemannian optimization since they allow to transfer data from a tangent space, which contains gradients of smooth functions as its elements, to the manifold, which is the search space in the optimization problem. The concept of retraction on the symplectic Stiefel manifold $\Spkn$ can be defined as follows. Let 
\[
\mathrm{T}\Spkn=\bigcup_{X\in\Spkn}\TU
\] 
be the tangent bundle to $\Spkn$. A~smooth mapping \mbox{$\mathcal{R}:\mathrm{T}\Spkn \to \Spkn$} is called a~\emph{retraction} if for all $X\in\Spkn$, the restriction of $\mathcal{R}$ to $\TU$, denoted by $\mathcal{R}_X$, 
satisfies the following properties:
\begin{enumerate}
	\item[1)] $\mathcal{R}_X(0_X)=X$, where $0_X$ denotes the origin of $\TU$;
	\item[2)] $\left.\tfrac{{\rm d}}{{\rm d}t} \mathcal{R}_X(tZ)\right|_{t=0}=Z$ for all $Z\in \TU$.
\end{enumerate}
The retraction provides a first-order approximation to the Riemannian exponential map 
\cite[Section~5.4]{AbsiMS08} which is of crucial importance in optimization algorithms on Riemannian manifolds. Since retractions are computationally less expensive compared to the exponential map while retaining the convergence properties of optimization schemes, they have 
attracted extensive interest in research. There are different approaches for the construction of retractions. In Subsections~\ref{ssec:Cayley} and~\ref{ssec:qgeo}, we briefly review the Cayley and quasi-geodesic retractions on $\Spkn$ first introduced in \cite{GSAS21} and then, in Subsection~\ref{ssec:SRretr}, we propose a~new one based on an~SR decomposition. 

\subsection{Cayley retraction}
\label{ssec:Cayley}

For $Z\in\TU$,  the \emph{Cayley retraction} on the symplectic Stiefel manifold $\Spkn$ is defined by
\begin{equation}\label{eq:retr_cayley}
\mathcal{R}_X^{\rm cay}(Z) := \Bigl(I_{2n}-\frac{1}{2}S_{X,Z}J_{2n}\Bigr)^{-1}\Bigl(I_{2n}+\frac{1}{2}S_{X,Z}J_{2n}\Bigr)X,
\end{equation}
where $S_{X,Z}=G_XZ(XJ_{2k})^T+XJ_{2k}(G_XZ)^T$ and
\mbox{$G_X=I_{2n}-\frac{1}{2}XJ_{2k}X^T\!J_{2n}^T$}, see \cite[Section~5.2]{GSAS21}. It exists if and only if $2$ is not an~eigenvalue of the Hamiltonian matrix $S_{X,Z}J_{2n}$.
The computation of the Cayley retraction \eqref{eq:retr_cayley} involves a matrix inverse of size $2n\times 2n$. In fact, it can be economically computed by the Sherman--Morrison--Woodbury formula as discussed in \cite[Proposition~5.5]{GSAS21} which requires the inversion of a~$4k\times 4k$ matrix. 
This is certainly advantageous since in most applications, $k$ is considerably smaller than~$n$.
In \cite[Proposition~5.2]{BendZ21}, an equivalent expression for the Cayley retraction 
\[
\mathcal{R}_X^{\rm cay}(Z) = -X+\bigl(H_{X,Z}+2X\bigr)\Bigl(\frac{1}{4}J_{2k}^TH_{X,Z}^TJ_{2n}^{}H_{X,Z}^{}-\frac{1}{2}J_{2k}^TX^TJ_{2n}^{}Z+I_{2k}\Bigr)^{-1}
\]
with $H_{X,Z}=Z-XJ_{2k}^TX^TJ_{2n}^{}Z$ has been presented, which requires solving a~linear system with a~$2k\times 2k$ matrix only. Note that this expression coincides with that 
considered in \cite[Lemma~3.1]{OviH21}. 

\subsection{Quasi-geodesic retraction}
\label{ssec:qgeo}

Given $Z\in \TU$, the \emph{quasi-geodesic retraction} is defined by
\begin{equation}\label{eq:Rqgeo}
\mathcal{R}^{\qgeo}_X(Z) := \fkh{X,\; Z} \exp\dkh{\begin{bmatrix}
	-J_{2k}W &  J_{2k}Z^T J_{2n} Z \\ I_{2k} & -J_{2k}W
	\end{bmatrix}} 
\begin{bmatrix}
I_{2k} \\ 0
\end{bmatrix} 
\exp(J_{2k}W), 
\end{equation}
where  $W=X^T J_{2n}Z$ and $\exp(\cdot)$ denotes the matrix exponential~\cite[Section~5.1]{GSAS21}. 
This retraction is globally defined. Note that the calculation of \eqref{eq:Rqgeo} requires computing two exponentials of matrices of size $4k\times 4k$ and $2k\times 2k$ which is dominating for relatively large~$k$.

\subsection{SR decomposition-based retraction}
\label{ssec:SRretr}

For some matrix manifolds, retractions can also be defined using related matrix decompositions,
e.g.,~\cite{AbsiMS08}.
In the case of the symplectic Stiefel manifold $\Spkn$, the SR decomposition introduced first in \cite{DDora75} appears to be very useful to define a new decomposition-based retraction. 

To begin with, we consider the {\em perfect shuffle permutation matrix}
\begin{equation}\label{eq:P2k}
P_{2k}^{ } = [e_1, e_3,\ldots,e_{2k-1},e_2,\ldots,e_{2k}],
\end{equation} 
where $e_j$, $j=1,\ldots,2k$, is the $j$-th canonical basis vector of $\mathbb{R}^{2k}$. It can be shown by direct calculation that this matrix is orthogonal and 
\begin{equation}
P_{2k}^{}J_{2k}^{}P_{2k}^T  = \diag(J_{2},\ldots,J_2) =:\hat{J}_{2k},
\label{eq:P2kJ2k}
\end{equation} 
where $P_{2k}^T=[e_1, e_{k+1},e_2, e_{k+2},\ldots,e_k,e_{2k}]$.
Further, we introduce a congruence matrix set
\[
T_{2k}(P_{2k})= \{P_{2k}^T\hat{R}P^{ }_{2k}\enskip:\enskip \hat{R}\in \mathbb{R}^{2k\times 2k} \;\text{ is upper triangular}\}.
\]
Then an~\emph{SR decomposition} of $A\in\mathbb{R}^{2n\times 2k}$ with $k\leq n$ is defined as  
\begin{equation}\label{Eq:SRdecomp}
A = SR,
\end{equation}
where $S\in \Spkn$ and $R\in T_{2k}(P_{2k})$. The existence of such a~decomposition for square matrices has been established in \cite[Theorem 3.8]{BunG86}. 
This result can be adapted to rectangular matrices as follows.

\begin{theorem}\label{th:existenceSR}
	Let $A\in\mathbb{R}^{2n\times 2k}$ have full column rank and let $P_{2k}$ be as in~\eqref{eq:P2k}. There exists an SR decomposition $A=SR$ with $S\in \Spkn$ and $R\in T_{2k}(P_{2k})$ if and only if all leading minors of even dimension of the matrix $P_{2k}^{}A^TJ_{2n}AP_{2k}^T$ are nonzero, i.e., $\mathrm{det}(P_{2k}A^TJ_{2n}AP^T_{2k})_{1:2j,1:2j}\neq 0$ for  $j=1,\dots,k$.
\end{theorem}

\begin{proof}
	Assume that $A$ has an~SR decomposition $A=SR$ with $S\in \Spkn$ and $R\in T_{2k}(P_{2k})$. Then we obtain
	\begin{align*}
	P_{2k}^{}A^TJ_{2n}AP_{2k}^T & = P_{2k}^{}R^TS^TJ_{2n}SRP_{2k}^T = P_{2k}^{}R^TJ_{2k}RP_{2k}^T
	= \hat{R}^T\hat{J}_{2k}\hat{R},
	\end{align*}
	where $\hat{R}=P_{2k}^{}RP_{2k}^T$ is upper triangular and $\hat{J}_{2k}$ is as in \eqref{eq:P2kJ2k}. Since $A$ has full column rank,  $R$ and therefore $\hat{R}$ are nonsingular.  In view of these facts, for $j=1,\ldots,k$, the $2j\times 2j$ leading principal submatrices of $\hat{R}^T\hat{J}_{2k}\hat{R}$ have the form 
	\[
	(\hat{R}^T\hat{J}_{2k}\hat{R})_{1:2j,1:2j}=(\hat{R}^T)_{1:2j,1:2j} (\hat{J}_{2k})_{1:2j,1:2j}(\hat{R})_{1:2j,1:2j}
	\]
	and, hence, they are, as a product of nonsingular matrices, nonsingular. Thus,  all leading minors of even dimension of $P_{2k}^{}A^TJ_{2n}AP_{2k}^T$ are nonzero.
	
	On the other hand, if all leading minors of even dimension of the skew-symmetric matrix $P_{2k}^{}A^TJ_{2n}AP_{2k}^T$ are nonzero, then by \cite[Theorem~2.2]{BenBFMW00} this matrix has the Cholesky-like decomposition \begin{equation}\label{eq:CholDec}
	P_{2k}^{}A^TJ_{2n}AP_{2k}^T=\hat{R}^T\hat{J}_{2k}\hat{R}
	\end{equation}
	with a nonsingular upper triangular matrix~$\hat{R}$. In this case, we have
	\[
	A^TJ_{2n}A= P_{2k}^T\hat{R}^T\hat{J}_{2k}\hat{R} P_{2k}^{} 
	=R^TJ_{2k}R,
	\]
	where $R=P_{2k}^T\hat{R} P_{2k}^{}\in T_{2k}(P_{2k})$ is nonsingular. Furthermore, the matrix $S=AR^{-1}$ is symplectic due to
	$S^TJ_{2n}S=R^{-T}A^TJ_{2n}AR^{-1}=J_{2k}$. Thus, $A=SR$ is an SR decomposition.
\end{proof}

Similarly to the square case \cite[Remark 3.9]{BunG86}, the SR decomposition of rectangular matrices is non-unique.
Indeed, if $A$ is decomposed as in \eqref{Eq:SRdecomp}, then for any nonsingular diagonal matrix $D\in\mathbb{R}^{k\times k}$, 
$$
A = (S\diag(D,D^{-1}))(\diag(D^{-1},D)R)
$$
is also an~SR decomposition of $A$. 
The freedom of choice of the factors $S$ and $R$ in \eqref{Eq:SRdecomp} is usually exploited to improve numerical stability, see, e.g., \cite{FassR16,SalaAF08}. 
To serve the purpose of constructing a~retraction on $\Spkn$, we follow \cite{Mehr79} and restrict the factor~$R$ in~\eqref{Eq:SRdecomp} to the matrix set 
\begin{equation}\label{eq:T_2k_0}
\arraycolsep=2pt
\begin{array}{rl}
T_{2k}^0(P_{2k}) =\bigl\{P_{2k}^T\hat{R}P^{ }_{2k}\enskip:\enskip
&\hat{R}=[r_{ij}]\in \mathbb{R}^{2k\times 2k} \text{ is upper triangular with } r_{2j-1,2j} = 0,\bigr. \\ 
&  \bigl.  
\,r_{2j-1,2j-1} > 0, \text{ and } 
|r_{2j,2j}|=r_{2j-1,2j-1}   
 \text{ for } 
j = 1,\ldots, k\bigr\}.
\end{array}
\end{equation}
This choice guarantees the uniqueness of the resulting SR decomposition~\cite{Mehr79}. 
Note that this fact can also be inferred from \cite[Theorem~2.2]{BenBFMW00}, which establishes the uniqueness of the Cholesky-like decomposition \eqref{eq:CholDec}  with $\hat{R}$ as in~\eqref{eq:T_2k_0}.

In order to compute the SR decomposition \eqref{Eq:SRdecomp}, we employ a~symplectic Gram--Schmidt algorithm developed in~\cite{Sala05}. Based on Theorem~\ref{th:existenceSR}, we investigate the 
well-posedness of this algorithm which was not discussed in \cite{Sala05}. For ease of explanation, our consideration is divided into three steps.

First, the SR decomposition of two-column matrices is needed. Given a~matrix $A = [a_1, a_2] \in \mathbb{R}^{2n\times 2}$, 
we are looking for an~{\em elementary SR} (ESR) {\em decomposition} $A=SR$ with a~$2 \times 2$ upper triangular matrix 
\begin{equation}\label{eq:matrR}
R = \begin{bmatrix}
r_{11}&r_{12}\\0 & r_{22}
\end{bmatrix}
\end{equation}
and a~two-column symplectic matrix $S = [s_1, s_2]$. 
It is straightforward to verify that $S= [s_1, s_2]$ is symplectic if and only if
$s_1^TJ_{2n}s_2=1$.
Using this relation, we find that  $r_{11}r_{22} = a_1^TJ_{2n}a_2$
and $r_{12}$ is arbitrary. By Theorem~\ref{th:existenceSR}, the existence of the ESR decomposition $A=SR$ is equivalent to the condition $a_1^TJ_{2n}a_2 \not=0$ which is also known as the non-isotropy condition for the subspace spanned by $a_1$ and $a_2$. Requiring $R\in	T_{2}^0(P_{2})$, we obtain
\begin{equation}\label{Eq:mESRchoosingentries}
r_{12} = 0,\quad r_{11}=\sqrt{|a_1^TJ_{2n}a_2|},\quad r_{22}=\mbox{sign}(a_1^TJ_{2n}a_2)r_{11},
\end{equation}
where $\mbox{sign}(\cdot)$ denotes the sign of the corresponding value. The resulting decomposition is referred  to as the {\em diagonal elementary SR} (DESR) {\em decomposition} 
which can easily be seen to be unique. 
Note that the choice~\eqref{Eq:mESRchoosingentries} corresponds to the version ESR4 in \cite{FassR16} which was proven to yield $R$ with a~minimal condition number. For convenience, we summarize the computation of the DESR decomposition in Algorithm~\ref{Alg:DESR} and collect its properties in the following lemma. 

\begin{algorithm}[htbp]
	\caption{Diagonal elementary SR (DESR) decomposition}
	\label{Alg:DESR}
	\begin{algorithmic}[1]
		\REQUIRE $A =[a_1, a_2]\in \mathbb{R}^{2n\times 2}$.
		\ENSURE $S = [s_1,s_2]\in \mathrm{Sp}(2,2n)$ and $R=\diag(r_{11},r_{22})$ with $0<r_{11}=|r_{22}|$ such that $A=SR$.
		\STATE Compute $\omega = a_1^TJ_{2n}a_2$.
		\IF{$\omega \neq 0$} 
		\STATE Compute $
		\ r_{11}=\sqrt{|\omega|}$ and $r_{22}=\mathrm{sign}(\omega)r_{11}$.
		\STATE Compute $s_1 = a_1/r_{11}$ and $s_2 = a_2/r_{22}$.
		\ELSE 
		\STATE Error: the DESR decomposition does not exist. 
		\ENDIF
	\end{algorithmic}
\end{algorithm}

\begin{lemma}\label{lem:DESR}
	For a matrix $A=[a_1,a_2] \in \mathbb{R}^{2n\times 2}$ with $a_1^TJ_{2n}a_2^{}\neq 0$, Algorithm~\textup{\ref{Alg:DESR}} produces  a~unique DESR decomposition $A=SR$. In particular, if $A$ is symplectic, then $S=A$ and $R=I_2$.
\end{lemma}

Second, for $A\in \mathbb{R}^{2n\times 2k}$, we consider a decomposition
\begin{equation}\label{Eq:BSRdecomp}
A = \hat{S}\hat{R},
\end{equation}
where $\hat{S}\in\mathbb{R}^{2n\times 2k}$ is a~\textit{perfect shuffle permuted symplectic} (PSPS) \textit{matrix} satisfying
\begin{equation}
\hat{S}^TJ_{2n}\hat{S} = \hat{J}_{2k},
\label{eq:V}
\end{equation}
and 
\begin{equation} \label{eq:uptrR}
\hat{R} = 
\begin{bmatrix}
\hat{R}_{11}& \cdots&\hat{R}_{1k}\\ &\ddots&\vdots\\&&\hat{R}_{kk}
\end{bmatrix}
\end{equation}
is upper triangular with $\hat{R}_{ij}\in\mathbb{R}^{2\times 2}$ for $1\leq i\leq j\leq k$. Additionally, the blocks $\hat{R}_{jj}$, $1\leq j\leq k$, are assumed to be diagonal with diagonal elements ordered nonincreasingly and having the same absolute value.
Note that the matrix $R$ in \eqref{eq:matrR} with entries satisfying \eqref{Eq:mESRchoosingentries} has this structure.
It follows from \eqref{eq:P2kJ2k} and \eqref{eq:V} that $\hat{S}$ is PSPS if and only if $\hat{S}P_{2k}^{}$ is symplectic, which justifies the name PSPS. 
Moreover, if $\hat{S}$ is PSPS, the $2n\times 2$ blocks $\hat{S}_j$ of the matrix  $\hat{S}=[\hat{S}_1,\dots,\hat{S}_k]$ are symplectic. The following lemma establishes the existence and uniqueness of the decomposition~\eqref{Eq:BSRdecomp}.

\begin{lemma}\label{lem:existenceBSGS}
	Let $A\in\mathbb{R}^{2n\times 2k}$ have full column rank. 
	The decomposition \eqref{Eq:BSRdecomp} with a~PSPS matrix $\hat{S}$ and an~upper triangular matrix~$\hat{R}$ as in \eqref{eq:uptrR} exists and is unique if and only if all leading minors of even dimension of the matrix $A^TJ_{2n}A$ are nonzero, i.e., $\mathrm{det}(A^TJ_{2n}A)_{1:2j,1:2j}\neq 0$ for  $j=1,\dots,k$.
\end{lemma}

\begin{proof}
	The necessity and sufficiency for the existence of the decomposition~\eqref{Eq:BSRdecomp}
	can be proved analogously to Theorem~\ref{th:existenceSR}. The special choice of the diagonal blocks of $\hat{R}$ guarantees the uniqueness.
\end{proof}

Partitioning $A =[A_1,\ldots, A_k]$ with $A_j\in\mathbb{R}^{2n\times 2}$
makes it possible to use the DESR decompositions for computing the $2n\times 2$ symplectic blocks $\hat{S}_j$ of 
$\hat{S} = [\hat{S}_1,\ldots,\hat{S}_k]$ and the $2\times 2$ blocks {$\hat{R}_{jj}$} of $\hat{R}$ in~\eqref{eq:uptrR}. To this end, we consider 
\[
\hat{J}_{2k}^T\hat{S}^TJ_{2n}^{}A= \hat{J}_{2k}^T \hat{S}^TJ_{2n}^{}\hat{S}\hat{R}=\hat{J}_{2k}^T\hat{J}_{2k}^{}\hat{R}=\hat{R}, 
\]
where equalities follow from \eqref{Eq:BSRdecomp} and \eqref{eq:V}. According to the block partition of~$A$ and $\hat{S}$, the block diagonal structure of $\hat{J}_{2k}$, and the block upper triangular structure of $\hat{R}$, we can work out a~block Gram--Schmidt-type algorithm for computing the decomposition~\eqref{Eq:BSRdecomp}, see Algorithm~\ref{Alg:BSGS}.

\begin{algorithm}[htbp]
	\caption{Basic symplectic Gram--Schmidt algorithm} 
	\begin{algorithmic}[1]
		\REQUIRE $A =[A_1,\ldots, A_k]\in \mathbb{R}^{2n\times 2k}$ with $A_j\in\mathbb{R}^{2n\times 2}$ for $j=1,\ldots,k$.
		\ENSURE A~PSPS matrix $\hat{S} = [\hat{S}_1,\ldots,\hat{S}_k]$ and an~upper triangular matrix~$\hat{R}$ as in~\eqref{eq:uptrR} such that $A=\hat{S}\hat{R}$.
		\STATE Compute the DESR decomposition $A_1 = \hat{S}_1\hat{R}_{11}$ using Algorithm~\ref{Alg:DESR}.
		\FOR{$j=2,\ldots,k$} 
		\FOR{$i=1,\ldots,j-1$}
		\STATE Compute $\hat{R}_{ij} = J_2^T \hat{S}_i^TJ_{2n}A_j$.
		\ENDFOR		
		\STATE Compute $W_j = A_j - \sum\limits_{i=1}^{j-1}\hat{S}_i\hat{R}_{ij}$.
		\STATE Compute the DESR decomposition $W_j = \hat{S}_j\hat{R}_{jj}$ using Algorithm~\ref{Alg:DESR}. 
		\ENDFOR
	\end{algorithmic}
	\label{Alg:BSGS}
\end{algorithm}

The following lemma provides the sufficient conditions for the existence of the DESR decomposition of the matrices $A_1$ and $W_j$, $j=2,\ldots, k$, in Algorithm~\ref{Alg:BSGS}.

\begin{lemma}\label{lem:BSGS}
	Let $A=[A_1,\ldots,A_k]$ with $A_j\in\mathbb{R}^{2n\times 2}$ for $j=1,\ldots,k$ be such that all leading minors of even dimension of the matrix $A^TJ_{2n}A$ are nonzero. Then Algorithm~\textup{\ref{Alg:BSGS}} produces the decomposition \eqref{Eq:BSRdecomp} with a PSPS matrix $\hat{S}$ and 
	an~upper triangular matrix~$\hat{R}$ as in \eqref{eq:uptrR} without breakdown. In particular, if~$A$ is PSPS, then $\hat{S}=A$ and $\hat{R}=I_{2k}$.
\end{lemma}

\begin{proof}
	We show by finite induction that the matrices $W_1=A_1$, $W_2,\ldots, W_k$ in Algorithm~\ref{Alg:BSGS} have the DESR decomposition. By assumption, the matrix  $W_1^TJ_{2n}W_1^{}=A_1^TJ_{2n}A_1^{}$ is nonsingular, and hence by Lemma~\ref{lem:DESR}, the matrix $W_1$ has the DESR decomposition. Assume that $j-1$ steps in Algorithm~\ref{Alg:BSGS} have been executed. For simplicity reasons, we introduce the following matrices 
	$B_{2(j-1)}=[A_1,\ldots,A_{j-1}]$, $Q_{2(j-1)}=[\hat{S}_1,\ldots,\hat{S}_{j-1}]$, and 
	\begin{align*}
	\hat{R}_{2(j-1)} & = \begin{bmatrix} \hat{R}_{11} & \cdots & \hat{R}_{1,j-1}\\ & \ddots & \vdots \\ & & \hat{R}_{j-1,j-1}\end{bmatrix}.
	\end{align*}
	Note that by construction, $\hat{R}_{2(j-1)}$ is nonsingular. Using the expressions for $\hat{R}_{ij}$ and $W_j$ in steps~4 and 6 of Algorithm~\ref{Alg:BSGS}, respectively, the matrix $W_j$ can be represented as 
	$W_j^{}=A_j^{}-Q_{2(j-1)}^{}\hat{J}_{2(j-1)}^TQ_{2(j-1)}^TJ_{2n}^{}A_j^{}$. Then we have
	\begin{align*}
	W_j^TJ_{2n}^{}W_j^{}& =A_j^TJ_{2n}^{}A_j^{}-2A_j^TJ_{2n}^TQ_{2(j-1)}^{}\hat{J}_{2(j-1)}^{}Q_{2(j-1)}^TJ_{2n}^{}A_j \\
	&\quad + A_j^TJ_{2n}^TQ_{2(j-1)}^{}\hat{J}_{2(j-1)}^{}Q_{2(j-1)}^TJ_{2n}^{} Q_{2(j-1)}^{}\hat{J}_{2(j-1)}^TQ_{2(j-1)}^TJ_{2n}A_j \\
	& = A_j^TJ_{2n}^{}A_j^{}-A_j^TJ_{2n}^TQ_{2(j-1)}^{}\hat{J}_{2(j-1)}^{}Q_{2(j-1)}^TJ_{2n}^{}A_j^{}.
	\end{align*}
	It follows from $B_{2(j-1)}=Q_{2(j-1)}\hat{R}_{2(j-1)}$ that $$
	Q_{2(j-1)}^TJ_{2n}A_j=\hat{R}_{2(j-1)}^{-T}B_{2(j-1)}^TJ_{2n}A_j.
	$$
	Therefore,
	\begin{align*}
	W_j^TJ_{2n}^{}W_j^{}& = A_j^TJ_{2n}^{}A_j^{}\!-\!A_j^TJ_{2n}^TB_{2(j-1)}^{}\hat{R}_{2(j-1)}^{-1} \hat{J}_{2(j-1)}^{}\hat{R}_{2(j-1)}^{-T}B_{2(j-1)}^TJ_{2n}^{}A_j^{} \\
	& = A_j^TJ_{2n}^{}A_j^{} \!+\!\bigl(\!A_j^TJ_{2n}^TB_{2(j-1)}^{}\bigr)\! \bigl(\!\hat{R}_{2(j-1)}^{T}\hat{J}_{2(j-1)}^{}\hat{R}_{2(j-1)}^{}\bigr)^{-1}\!\bigl(\!B_{2(j-1)}^TJ_{2n}^{}A_j^{}\bigr) \\
	& = A_j^TJ_{2n}^{}A_j^{} \!+\!\bigl(\!A_j^TJ_{2n}^TB_{2(j-1)}^{}\bigr)\! \bigl(B_{2(j-1)}^{T}J_{2n}^{}B_{2(j-1)}^{}\bigr)^{-1}\!\bigl(\!B_{2(j-1)}^TJ_{2n}^{}A_j^{}\bigr)
	\end{align*}
	is the Schur complement of the block $B_{2(j-1)}^{T} J_{2n}^{}B_{2(j-1)}^{}$ of the matrix 
	\[
	\begin{bmatrix} B_{2(j-1)}^{T} J_{2n}^{}B_{2(j-1)}^{} &  B_{2(j-1)}^{T}J_{2n}^{}A_j^{} \\
	-A_j^TJ_{2n}^T B_{2(j-1)}^{T} & A_j^TJ_{2n}^{} A_j^{}\end{bmatrix} =[A_1,\ldots,A_j]^TJ_{2n}[A_1,\ldots,A_j].
	\]
	Since this matrix is nonsingular, the Schur complement $W_j^TJ_{2n}^{}W_j^{}$ is also nonsingular. Thus, by Lemma~\ref{lem:DESR}, $W_j$ has the DESR decomposition.   
	The particular case for a~PSPS matrix $A$ follows from the uniqueness of the decomposition~\eqref{Eq:BSRdecomp}.
\end{proof}

Finally, once the basic symplectic Gram--Schmidt algorithm with the DESR decompositions described above is well-defined, the computation of the SR decomposition~\eqref{Eq:SRdecomp}
requires just two extra permutation steps as presented in Algorithm~\ref{Alg:SGS}. 

\begin{algorithm}[htbp]
	\caption{Symplectic Gram--Schmidt algorithm} 
	\label{Alg:SGS}
	\begin{algorithmic}[1]
		\REQUIRE $A =[A_1,\ldots, A_k]\in \mathbb{R}^{2n\times 2k}$ with $A_j\in\mathbb{R}^{2n\times 2}$ for $j=1,\ldots,k$.
		\ENSURE $S \in \Spkn$ and $R\in T_{2k}^0(P_{2k})$ such that $A=SR$.
		\STATE Compute the decomposition $AP^T_{2k}=\hat{S}\hat{R}$ by  using Algorithm~\ref{Alg:BSGS}.
		\STATE Compute $S=\hat{S}P^{}_{2k}$.
		\STATE Compute $R = P^T_{2k}\hat{R}P^{}_{2k}$.
	\end{algorithmic}
\end{algorithm}

Using Lemma~\ref{lem:BSGS}, we can establish the following properties of~Algorithm~\ref{Alg:SGS}.

\begin{proposition}\label{prop:SGS}
	Let $A\in\mathbb{R}^{2n\times 2k}$ and let $P_{2k}$ be as in \eqref{eq:P2k}.
	Assume that all leading minors of even dimension of $P_{2k}^{}A^TJ_{2n}AP_{2k}^T$ are nonzero.
	Then Algorithm~\textup{\ref{Alg:SGS}} produces a~unique SR decomposition \eqref{Eq:SRdecomp} with $S\in\Spkn$ and $R\in T_{2k}^0(P_{2k})$.
\end{proposition}

\begin{proof}
	The existence and uniqueness of the SR decomposition~\eqref{Eq:SRdecomp} provided by  Algorithm~\ref{Alg:SGS} follow from that of the decomposition $AP^T_{2k}=\hat{S}\hat{R}$ with a~PSPS matrix~$\hat{S}$ and an upper triangular matrix $\hat{R}$ as in~\eqref{eq:uptrR}  which is guaranteed by Lemma~\ref{lem:BSGS}. Therefore, the decomposition $AP^T_{2k}=\hat{S}\hat{R} = \hat{S}P^{}_{2k}P^T_{2k}\hat{R}$ yields that $A=(\hat{S}P^{}_{2k})(P^T_{2k}\hat{R}P^{}_{2k}) = SR$,
	where  $S=\hat{S}P^{}_{2k}\in \Spkn$ and \mbox{$R=P^T_{2k}\hat{R}P^{}_{2k}\in T_{2k}^0(P_{2k})$}.
\end{proof}

\begin{remark}
	For clarity of theoretical discussion, we have presented here the basic symplectic Gram--Schmidt algorithm only. In practice, however, similarly to the Gram--Schmidt orthonormalization process~\textup{\cite{GoluV13}}, for a~better numerical behavior, a~modified basic symplectic Gram--Schmidt procedure should be used. This procedure leads to a~modified symplectic Gram--Schmidt algorithm, see \textup{\cite{Sala05}} for more details, which is indeed employed in numerical experiments reported in Section~\textup{\ref{sec:appl}}.
\end{remark}

We are now ready to introduce a new retraction on $\Spkn$ which is based on the SR decomposition.

\begin{theorem}\label{theo:SRretraction}
	Given $X \in \Spkn$ and a tangent vector $Z \in \TU$, denote by $\mathrm{sf}(X+Z)$ the factor $S\in \Spkn$ in the SR decomposition $X+Z=SR$ computed by Algorithm~\textup{\ref{Alg:SGS}}.
	Then the mapping
	\begin{equation}\label{eq:retraction}
	\mathcal{R}_X^{\rm SR}(Z) =\mathrm{sf}(X+Z)
	\end{equation}
	defines a~retraction on $\Spkn$.
\end{theorem}

\begin{proof}
	We consider the mapping 
	\begin{equation*}
	\arraycolsep=2pt
	\begin{array}{rcl}
	\Theta:\Spkn\times T_{2k}^0(P_{2k})  &\rightarrow & \mathcal{U} 
	\\
	(S,R)&\mapsto & SR,
	\end{array}
	\end{equation*}
	where $\mathcal{U}\subset\mathbb{R}_*^{2n\times 2k}$ consists of  
	matrices having an SR decomposition and $\mathbb{R}_*^{2n\times 2k}$ denotes the set of real $2n\times 2k$ matrices of full rank.   It follows from Theorem~\ref{th:existenceSR} that~$\,\mathcal{U}$ is an~open subset of $\mathbb{R}_*^{2n\times 2k}$. Moreover, $\Theta$ admits a~neutral element $I_{2k}\in T_{2k}^0(P_{2k})$ satisfying
	\[
	\Theta(S,I_{2k})=SI_{2k} = S \quad\mbox{ for all } S \in \Spkn.
	\]
	Since $\mathrm{dim}\bigl(\Spkn\bigr)=4nk-k(2k-1)$ and $\mathrm{dim}\bigl(T_{2k}^0(P_{2k})\bigr)=k(2k-1)$, we obtain that 
	$$
	\mathrm{dim}(\Spkn)+\mathrm{dim}(T_{2k}^0(P_{2k}))=4nk=\mathrm{dim}(\mathbb{R}_*^{2n\times 2k}).
	$$
	The mapping~$\Theta$ is smooth as it is just the matrix multiplication restricted to the submanifolds. Furthermore, the inverse mapping $\Theta^{-1}$ is defined on the whole $\,\mathcal{U}$. 	
	For any $A\in\mathcal{U}$, the first component~$S$ of $\Theta^{-1}(A)$ is obtained by Algorithm~\ref{Alg:SGS} applied to $A$. Since this algorithm consists of basic mathematical operations only,  it is smooth on $\mathcal{U}$. The second component of~$\Theta^{-1}(A)$ is determined as $R = S^+A$, where $S^+ = J_{2k}^TS^TJ_{2n}$ is the symplectic inverse of $S$. This means that $\Theta$ is indeed a~diffeomorphism. Then by \cite[Proposition~4.1.2]{AbsiMS08}, the mapping 
	$\mathcal{R}_X^{\rm SR}$ in \eqref{eq:retraction} is a~retraction on $\Spkn$.	
\end{proof}

For the global convergence of a Riemannian gradient-based optimization algorithm that employs this retraction, its domain is crucial. Given $X\in\Spkn$, using Theorems~\ref{th:existenceSR} and \ref{theo:SRretraction}, we can deduce that the set of the tangent vectors~$Z$, for which the retraction $\mR_X^{\rm SR}(Z)$ does not exist, has measure zero. 
This fact is unfortunately not enough to guarantee the convergence of the associated algorithm as it requires that the retraction is locally well-defined around the origin~$0_X$ in the tangent space, i.e., the domain of the retraction should contain an open ball centered at~$0_X$ in the tangent space $\TU$, see \cite[Theorem 5.7]{GSAS21}. For the SR retraction developed here, the following theorem indicates that the statement holds for the ball of radius~one.

\begin{theorem}\label{theo:SRretraction_existence}
	Let $X \in \Spkn$. If a tangent vector $Z \in \TU$ satisfies $\|Z\|_2<1$, then $X+Z$ has an~SR decomposition.
\end{theorem}

\begin{proof}
	We will show that all leading minors of even dimension of the matrix $P_{2k}^{}(X+Z)^TJ_{2n}^{}(X+Z)P_{2k}^T$ are nonzero. Taking \eqref{eq:tangsp_Euclidean1},  the orthogonality of~$P_{2k}$, and  \eqref{eq:P2kJ2k} into account, we obtain
	\begin{align*}
	P_{2k}^{}(X\!+\!Z)^TJ_{2n}^{}(X\!+\!Z)P_{2k}^T & = P_{2k}^{}(X^T\!J_{2n}^{}X\!+\!Z^T\!J_{2n}X\!+\!X^T\!J_{2n}^{}Z\!+\!Z^T\!J_{2n}^{}Z)P_{2k}^T \\
	& = P_{2k}^{}(J_{2k}^{}+Z^T\!J_{2n}^{}Z)P_{2k}^T\\
	& = P_{2k}^{}J_{2k}^{}P_{2k}^T\big(I_{2k}^{}-(P_{2k}^{}J_{2k}^{}P_{2k}^T)(P_{2k}^{}Z^T\!J_{2n}^{}ZP_{2k}^T)\big)\\
	& = \hat{J}_{2k}^{}\big(I_{2k}^{}-\hat{J}_{2k}^{}(P_{2k}^{}Z^T\!J_{2n}^{}ZP_{2k}^T)\big).
	\end{align*}
	Because of the special structure of $\hat{J}_{2k}^{}$, the nonsingularity of 
	\[
	\Big(P_{2k}^{}(X+Z)^TJ_{2n}^{}(X+Z)P_{2k}^T\Big)_{1:2j,1:2j}\! = \!
	\Big(\hat{J}_{2k}^{}\big(I_{2k}^{}- \hat{J}_{2k}^{}(P_{2k}^{}Z^T\!J_{2n}^{}ZP_{2k}^T)\big)\Big)_{1:2j,1:2j}
	\]
	is the same as that of $\big(I_{2k}^{}-\hat{J}_{2k}^{}(P_{2k}^{}Z^T\!J_{2n}^{}ZP_{2k}^T)\big)_{1:2j,1:2j}$ for $j = 1,\ldots,k$. The latter in turn can be shown to be true. Indeed, 
	since $\|J_{2n}\|_2=\|\hat{J}_{2k}\|_2=1$ and $P_{2k}$ is orthogonal, we have
	\begin{align*}
	\Big\|\big(\hat{J}_{2k}^{}(P_{2k}^{}Z^T\!J_{2n}^{}ZP_{2k}^T)\big)_{1:2j,1:2j}\Big\|_2
	&=\Big\|(\hat{J}_{2k}^{})_{1:2j,1:2j}(P_{2k}^{}Z^T\!J_{2n}^{}ZP_{2k}^T)_{1:2j,1:2j}\Big\|_2\\
	& \leq \Big\|(P_{2k}^{}Z^T\!J_{2n}^{}ZP_{2k}^T)_{1:2j,1:2j}\Big\|_2 \\
	& \leq\big \|P_{2k}^{}Z^T\!J_{2n}^{}ZP_{2k}^T\big\|_2 \leq\big \|Z\big\|_2^2<1
	\end{align*}
	for $j=1,\ldots,k$. Thus, in view of Theorem~\ref{th:existenceSR}, the proof is complete.
\end{proof}

\section{Riemannian gradient method with non-monotone line search}
\label{sec:linesearch}

We now present the Riemannian gradient method with non-monotone line search for solving the constrained optimization problem \eqref{eq:min}. Starting with an~initial guess \mbox{$X_0\in\Spkn$}, this method generates a~sequence of iterates $\{X_i\}$ using a~search direction \mbox{$-\grad f(X_i)\in \mathrm{T}_{X_i}\Spkn$} as
\[
X_{i+1} = \mathcal{R}_{X_i}(-\tau_i\,\grad f(X_i)),
\]
where $\mathcal{R}_{X_i}$ is one of the retractions defined in Section~\ref{sec:retraction},  and $\tau_i>0$ is an~appropriate step size. We summarize the resulting Riemannian gradient method combined with the alternating Barzilai--Borwein strategy \cite{BarB88} for the  step size in Algorithm~\ref{alg:non-monotone gradient}. 

\begin{algorithm}[htbp]
	\caption{Riemannian gradient method for the optimization problem~\eqref{eq:min}}
	\label{alg:non-monotone gradient}
	\begin{algorithmic}[1]
		\REQUIRE The cost function $f$, metric $g$ and retraction $\mathcal{R}$ on $\Spkn$, 
		initial guess $X_0\in\Spkn$, $\gamma_0>0$,
		\mbox{$0<\gamma_{\min}<\gamma_{\max}$}, 
		$\beta, \delta\in(0,1)$, $\alpha \in [0,1]$, 
		$q_0=1$, $c_0 = f(X_0)$. 
		\ENSURE Sequence of iterates $\{X_i\}$. 
		\FOR{$i=0,1,2,\dots$}
		\STATE Compute $Z_i = -\grad f(X_i)$. 
		\IF{$i>0$} 
		\STATE{ 
			$
			\gamma_i=\left\{\begin{array}{ll}
			\dfrac{\|W_{i-1}\|_F^2}{\abs{\tr(W_{i-1}^T Y_{i-1}^{})}} 
			&\text{for odd } i, \\[3mm]
			\dfrac{\abs{\tr(W_{i-1}^T Y_{i-1}^{})}}{\|Y_{i-1}\|_F^2}
			&\text{for even } i,
			\end{array}\right.
			$\\
			where $W_{i-1} = X_i - X_{i-1}$ and $Y_{i-1} =Z_i-Z_{i-1}$.
		}
		\ENDIF
		\STATE Calculate the trial step size $\gamma_i=\max\bigl(\gamma_{\min},\min(\gamma_i,\gamma_{\max})\bigr)$.			
		\STATE Find the smallest integer $\ell$  such that the non-monotone condition 
		\[
		f\big(\mathcal{R}_{X_i}(\tau_i Z_i)\big) \le c_m + \beta\, \tau_i\, g\big(\grad f(X_i), Z_i\big)
		\]
		holds, where $\tau_i=\gamma_i\, \delta^{\ell}$. 
		\STATE Set $X_{i+1} = \mathcal{R}_{X_i}(\tau_i Z_i)$.
		\STATE Update $q_{i+1} = \alpha q_{i} + 1$ and
		$\displaystyle{c_{i+1} = \frac{\alpha q_{i}}{q_{i+1}} c_{i}  + \frac{1}{q_{i+1}} f(X_{i+1})}$.
		\ENDFOR
	\end{algorithmic}
\end{algorithm}

It has been shown in \cite[Theorem~5.7]{GSAS21} that every accumulation point~$X_*$ of the sequence~$\{X_i\}$ generated by Algorithm~\ref{alg:non-monotone gradient} is a~critical point of the cost function~$f$ in~\eqref{eq:min}, i.e., $\grad f(X_*)=0$, no matter which metric~$g$ and which retraction $\mathcal{R}$ are employed provided that 
$0_{X_*}\in\mathrm{T}_{X_*}\Spkn$ is in the interior of the domain of~$\mathcal{R}$.  
Note that due to Theorem~\ref{theo:SRretraction_existence}, the new  retraction~\eqref{eq:retraction} satisfies this condition.

\section{Applications}
\label{sec:appl}

Algorithm~\ref{alg:non-monotone gradient} is presented without any specification of metric nor retraction. As discussed in Sections~\ref{sec:Stiefel} and \ref{sec:retraction}, we can take either the canonical-like~(C) 
or Euclidean~(E) metric and one of the retractions based on Cayley transform (Cayley), quasi-geodesics (QGeo)
and SR decomposition (SR). 
It is important to note that they are independently constructed and can freely be combined.
As a result, we obtain totally six optimization schemes which are referred to as
CayleyC, CayleyE, QGeoC, QGeoE, SRC, and SRE. In this section, we present different applications for optimization on the symplectic Stiefel manifold $\Spkn$ and compare the numerical performance of different optimization schemes applied to various problems and scenarios. Three quantities of most interest are the values of the cost function, the norms of the Riemannian gradient of this function, and the feasibility violation of the approximate optimal solution. Depending on the problem, other performance and error measures will also be investigated. The numerical experiments are performed on a~standard desktop with Intel(R) Core(TM) i9-11900K (at 3.50GHz, 16MB Cache, 32GB RAM) running MATLAB R2021b under Ubuntu 22.04. The code including different optimization schemes is available from \href{https://github.com/opt-gaobin/spopt}{https://github.com/opt-gaobin/spopt}. 

Various parameters in Algorithm~\ref{alg:non-monotone gradient} need to be set. Most of them are fixed for all problems except for the maximal number of iterations $\texttt{niter}$ and the tolerance $\texttt{gtol}$ in the stopping criterion $\|\grad f(X_i)\|_F\leq\texttt{gtol}$.
For the backtracking search, we set $\beta = 1\e{-4}$, $\delta = 1\e{-1}$,    $\gamma_0 = 1\e{-3}$ (if it is not otherwise specified), $\gamma_{\min} = 1\e{-15}$, $\gamma_{\max} =1$ for the symplectic eigenvalue computation problem  and $\gamma_{\max} = 1\e\!+\!5$ for the rest. For the non-monotone condition, we choose $\alpha = 0.85$. Further, we use the canonical-like metric~$g_{c,\rho}$ with $\rho=1/2$ as recommended in \cite[Subsection 6.2]{GSAS21}. Note also that Theorem~\ref{theo:SRretraction_existence} provides a sufficient condition on
the size of $\|\tau_iZ_i\|_2$ in step 8 of Algorithm~\ref{alg:non-monotone gradient} 
in the case of using the SR~retraction. In our experiments, we never experienced a problem with the SR~decomposition even without any adjustment of~$\tau_i$  for this existence purpose.

\subsection{Symplectic target matrix problem}\label{ssec:Example_sympl_target}
In this problem, one has to solve the following minimization problem
\begin{equation} \label{eq:nearest_prob}
\min_{X\in \Spn}f(X) := \|X-W\|^2_F,
\end{equation}
where $W \in \Spn$ is given. 
In the \emph{optimal control of (symplectic) quantum gate}, it is shown that one can realize a target/ideal quantum gate with symplectic transformations by minimizing the distance between the real and ideal gates which in turn can be simplified to solving~\eqref{eq:nearest_prob}, where $W$ is a~matrix representation of the gate \cite{WuCR08,WuCR10}. A more general problem than \eqref{eq:nearest_prob}, where the cost function is given in the form of a~uniformly weighted sum of distances to several symplectic targets, has been investigated in \cite{Fior16}. An~extension
of this problem to the symplectic Stiefel manifold $\Spkn$ and a~general target matrix $W\in\R^{2n\times 2k}$, termed as the \emph{nearest symplectic matrix problem}, has been considered as a~test problem in~\cite{GSAS21}. Here, we examine two test examples. In the first one, the SUM gate
is used as the symplectic target gate, which is given by 
\begin{equation*}
W = \left[\begin{array}{c c c r}
1&0&0&0\\
1 &1& 0&0\\
0&0 &1&-1\\
0&0&0&1
\end{array}\right],
\end{equation*}
see \cite{WuCR10} and references therein. We run two experiments with different initial guesses and the same maximal number of iterations $\texttt{niter} = 500$ and tolerance $\texttt{gtol} = 1\e{-12}$.
First, we choose $X_0 = I_4$ as the starting point. Figure~\ref{fig:near_quangate_compare} demonstrates that all six~optimization schemes converge to the global minimizer $W$.  Notably, the ones with the Euclidean metric get to the target after only a~few steps while those with the canonical-like metric need more iterations. Moreover, the symplecticity of their iterates is truly attained which makes the corresponding plots disappear in the right subfigure in Figure~\ref{fig:near_quangate_compare}.  

For another initial guess $X_0 = \diag(1.728, -1.2, 1/1.728, -1/1.2)$, the convergence results and the feasibility violation are presented in Figure~\ref{fig:near_quangate2_compare}. One can see that all optimization schemes converge not to the minimizer but to a~critical point which, in view of \cite[Theorem 4.2]{WuCR10}, is a~saddle point.
The effect of different metrics on the numerical behavior, though not considerably as the first run, but still can be seen: the ones with the Euclidean metric apparently converge faster.

In the second example, we employ artificial data in order to check different scenarios and problem sizes. To this end, we choose the symplectic target matrix $W = \left[\begin{smallmatrix}
I&0\\V&I
\end{smallmatrix}\right]$ and the initial iterate $X_0 =\left[\begin{smallmatrix}
I&Y\\0&I
\end{smallmatrix}\right]$ with randomly generated matrices \mbox{$V, Y \in \symset(200)$}, see \cite[Lemma 2.1]{DopiJ09}. We run Algorithm~\ref{alg:non-monotone gradient} for $\texttt{niter} =1000$ iterations with the tolerance $\texttt{gtol} = 1\e{-10}$. Figure~\ref{fig:near_artificialdata_compare} shows that the difference between the optimization schemes based on the canonical-like and Euclidean metric is even severer: the ones with the canonical-like  metric do not provide reasonable results after $1000$ iterations, while the others reach the tolerance after less than  $50$~iterations. Comparing the convergence results for the latter in the lower row  of Figure~\ref{fig:near_artificialdata_compare}, we observe that the SRE scheme has better performance than CayleyE and QGeoE. 
Replacing the initial guess~$X_0$ with another symplectic matrices listed in  
\cite[Lemma~2.1]{DopiJ09}, we obtain the similar results.

These experiments apparently tell us that if the cost function expresses the distance resulting from the Frobenius norm, which can also be thought of as the Euclidean norm in the corresponding matrix space, the Riemannian optimization schemes with the Euclidean metric are preferred. Moreover, even in the convergent case, the computed critical point can be a~saddle point which is far from being satisfactory.

\begin{figure}[htbp]
	\centering
	\includegraphics[width=0.9\textwidth]{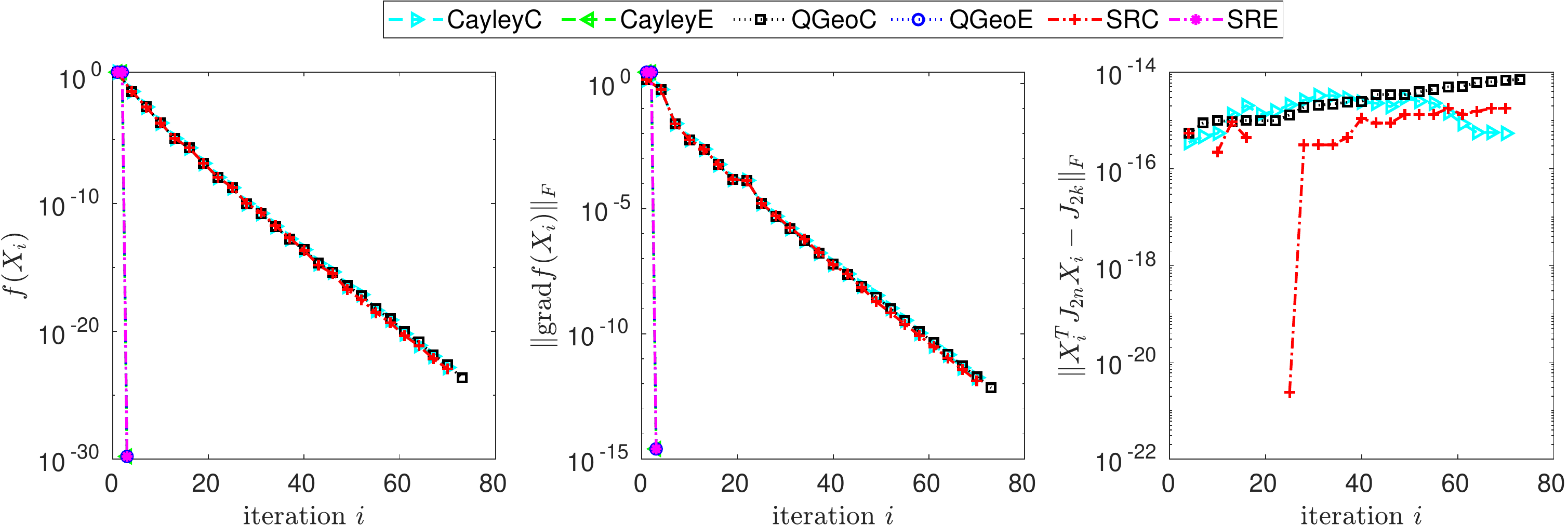}
	{\caption{Optimal control of the SUM gate with convergence to the global minimizer: a~comparison of the cost function values (left), Riemannian gradient norms (middle), and the feasibility violation (right) for different optimization schemes. 
		}
		\label{fig:near_quangate_compare}}
\end{figure}

\begin{figure}[htbp]%
	\centering
	\includegraphics[width=0.9\textwidth]{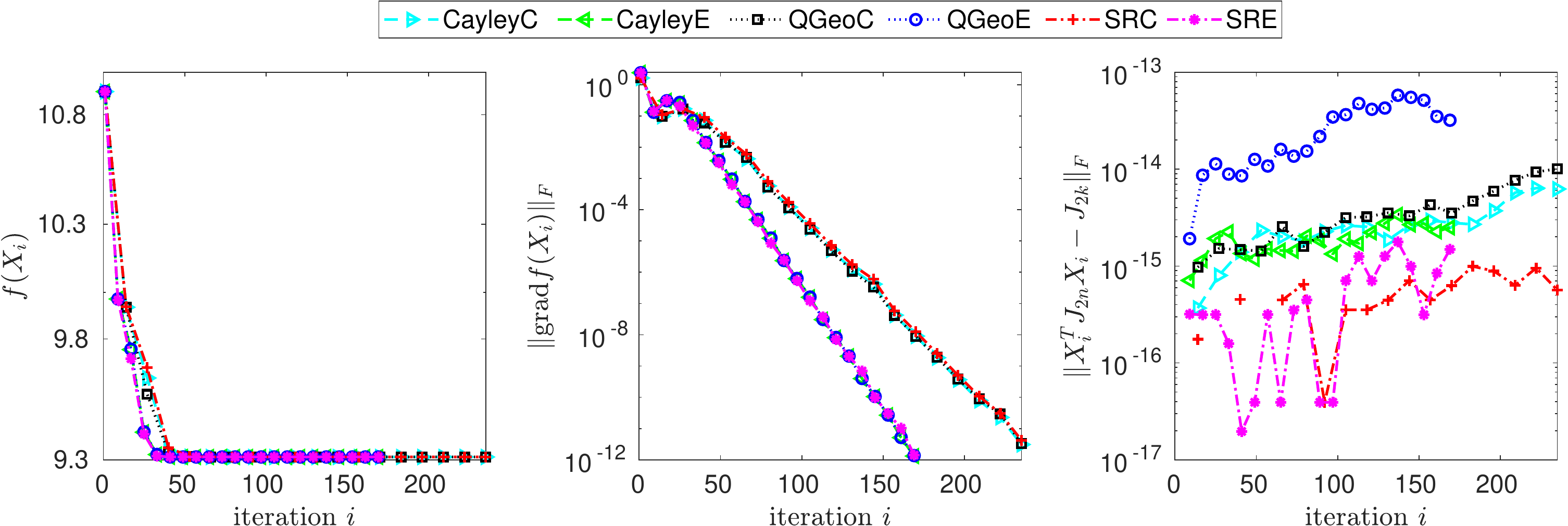}
	{\caption{Optimal control of the SUM gate with convergence to a saddle point: a comparison of the cost function values (left), Riemannian gradient norms (middle), and the feasibility violation (right) for different optimization schemes.
		}
		\label{fig:near_quangate2_compare}}
\end{figure}

\begin{figure}[htbp]%
	\centering
	\includegraphics[width=0.9\textwidth]{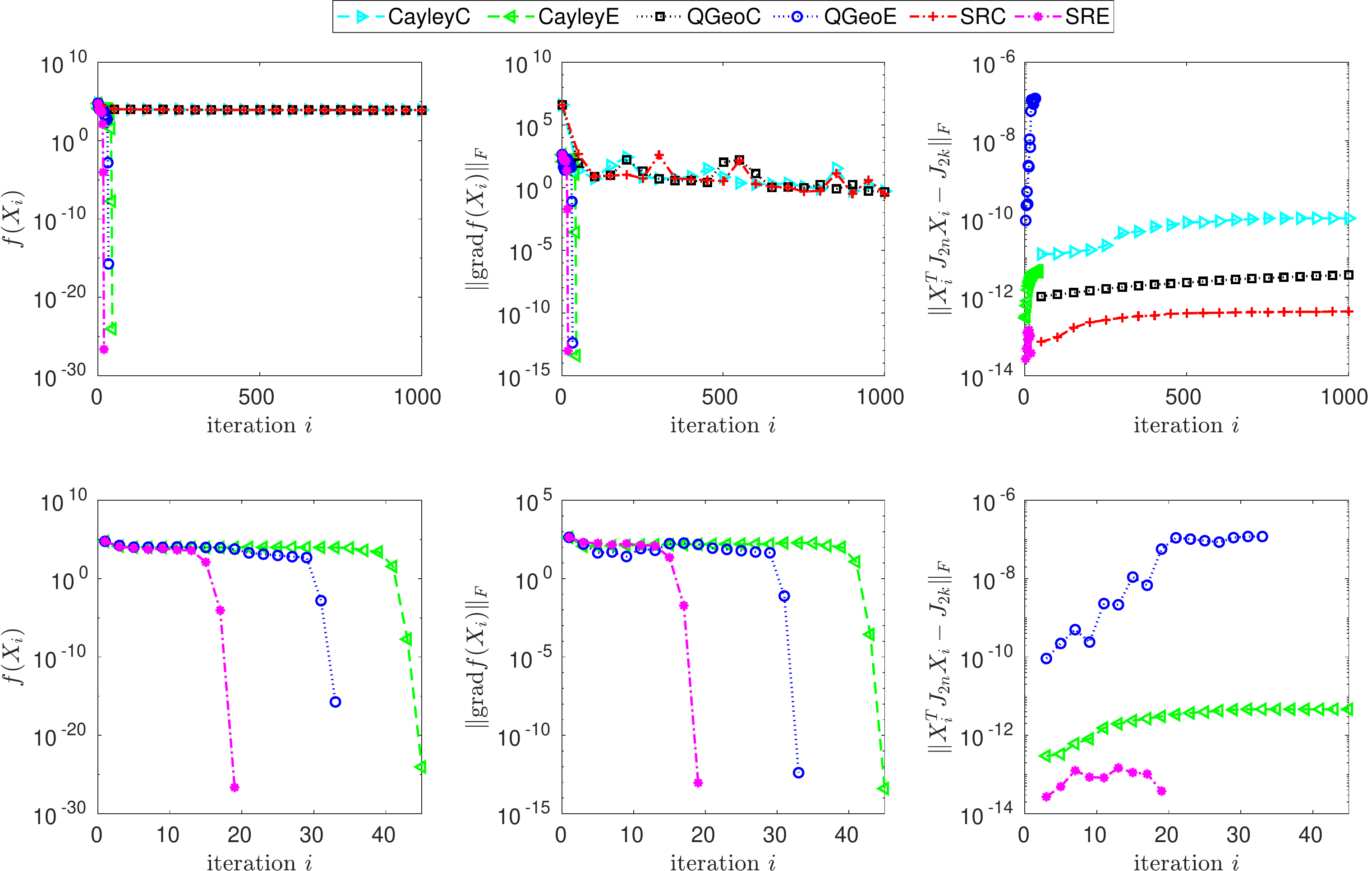}
	\caption{A symplectic target matrix problem: a comparison of the cost function values (left), Riemannian gradient norms (middle), and the feasibility violation (right) for six different optimization schemes (upper row) and three schemes with the Euclidean metric (lower row).
	}
	\label{fig:near_artificialdata_compare}
\end{figure}

\subsection{Computing symplectic eigenvalues}
\label{ssec:symplEV}

Originated from Williamson's work~\cite{Will36}, for a $2n\times 2n$ symmetric positive-definite (SPD) real matrix $A$, there exists a matrix $S\in \Spn$ such that
\begin{equation}\label{eq:sympleigdecomp}
S^TAS = \begin{bmatrix}
D & 0 \\ 0 & D
\end{bmatrix},
\end{equation}
where $D=\diag(d_1,\ldots,d_n)$ is a diagonal matrix with positive diagonal entries. The right-hand side of~\eqref{eq:sympleigdecomp} is termed as \emph{Williamson's diagonal form} of~$A$, and the positive numbers $d_1,\ldots,d_n$ are referred to as \emph{symplectic eigenvalues} of $A$. For ease of later argument, the symplectic eigenvalues are always numbered in the nondecreasing order, i.e., $d_1 \leq \cdots \leq d_n$.
Note that the symplectic eigenvalues of $A$ differ from the standard eigenvalues of $A$ but strongly relate to the eigenvalues of the Hamiltonian matrix $J_{2n}A$  or the symmetric/skew-symmetric matrix pencil $A-\lambda J_{2n}$, 
and the Hermitian matrix $\im A^{1/2}J_{2n}A^{1/2}$ with $\im =\sqrt{-1}$ and $A^{1/2}$ being the symmetric square root of~$A$ or the Hermitian pencil $A-\lambda\im J_{2n}$, which have been intensively investigated in \cite{Amod06,LancR06,BhatJ15,JainM22}, to name a few. 
A~\emph{pair of symplectic eigenvectors} $u, v \in \Rn\backslash\{0\}$ associated with a~symplectic eigenvalue~$d$ of $A$ are those that satisfy
\begin{equation*}
Au = d\, J_{2n}v,\qquad Av = -d\, J_{2n}u.
\end{equation*}
Symplectic eigenvalues and eigenvectors can be numerically computed  
using a~symplectic Lanczos method via the connection with  so-called positive-definite Hamiltonian matrices \cite{Amod06} or by solving a trace minimization problem  
\begin{equation}\label{eq:tracetheorem}
\min_{X\in \Spkn}f(X):= \tr(X^TAX) 
\end{equation}
using a Riemannian optimization method \cite{SonAGS21}. It has been shown in \cite{Hiro06,BhatJ15} that
the minimal value of $f$ in~\eqref{eq:tracetheorem} is twice the sum of $k$ smallest symplectic eigenvalues of $A$.
It is worth to note that this result can also be derived based on the trace minimization theorems, called the Ky-Fan theorem, for standard eigenvalues of real symmetric positive-semidefinite pencils \cite{KovaV95,Nakic03} or its extended complex Hermitian version \cite{LianLB13}. Recently, in \cite{SonSt22, LianWZL23}, the notion of symplectic eigenvalues and its trace minimization theorem have been extended to a~special class of symmetric positive-semidefinite (SPSD) matrices which have symplectic null space. Symplectic eigenvalues find applications in quantum mechanics, optics, stability analysis of gyroscopic systems, and in quantization process of superconducting networks modeled by Hamiltonian systems, see \cite{Hiro06,BennFS08,Lanc13,BuiFK14,KrbeTV14,EgusP22}. 

In the optimization approach, to compute the  $k$ smallest symplectic eigenvalues with $k$ pairs of symplectic eigenvectors for $1 \leq k \leq n$,  one first finds a~minimizer \mbox{$X_*\in\Spkn$} of the cost function $f$ in \eqref{eq:tracetheorem} and then 
diagonalizes $X_*^TAX_*^{}$ by an~orthosymplectic matrix $K\in\mathbb{R}^{2k\times 2k}$ such that $K^TX_*^TAX_*^{}K = \diag(\Lambda,\Lambda)$ with $\Lambda=\diag(\lambda_1,\ldots,\lambda_k)$. It has been shown in \cite{SonAGS21} 
that $\lambda_1,\ldots,\lambda_k$ and the pair of the $j$-th and $(j+k)$-th columns of~$X_*K$, $j = 1,\ldots,k$, are the sought symplectic eigenvalues and associated symplectic eigenvector pairs of~$A$, respectively. 

In the first example, in view of \cite{SonSt22}, we construct an~SPSD matrix 
\[
A = J_{2n}Q\begin{bmatrix}
D & 0 \\ 0 & D
\end{bmatrix}(J_{2n}Q)^T
\] 
with a symplectic null space by setting 
$$
D = \diag(0,\ldots,0,m+1,m+2,\ldots,n-1,n)
$$ 
with $0 < m < k < n$ and $Q = KL(n/5, 1.2,-\sqrt{n/5})$, where $L(l,c,d)$ is the (non-orthogonal) symplectic Gauss transformation 
\[
L(l,c,d) = \begin{bmatrix}
I_{l-2}& & & & & & &\\
& c & & & & & d &\\
& & c & & & d & & \\
& & & I_{n-l} & & & & \\
& & & & I_{l-2} & & & \\
& & & & & c^{-1}& & \\
& & & & & & c^{-1}& \\
& & & & & & & I_{n-l}\\
\end{bmatrix}
\]
and $K\! =\! \left[\!\begin{smallmatrix}\enskip\;\real(U)& \imag(U)\\-\imag(U)&\real(U)
\end{smallmatrix}\!\right]\ts\!{\in\Spn}$ with a~randomly generated unitary matrix 
\mbox{$U\!\in\! \mathbb{C}^{n\times n}$}.
Obviously,  $A$ is of rank $2(n-m)$ and its null space has dimension~$2m$. In our experiments, we take 
$n = 1000$, $k = 5$, and $m = 2$. In this case, the $5$~smallest symplectic eigenvalues of $A$ are $0, 0, 3, 4, 5$. The accuracy of the computed symplectic eigenvalues~$\tilde{d}_j$, $j=1,\ldots,k$, is verified by 
the \mbox{$l_1$-norm} error defined as $\sum_{j=1}^{k}|d_j-\tilde{d}_j|$. For all six optimization schemes, we run Algorithm~\ref{alg:non-monotone gradient} with $\texttt{niter}=5000$  iterations and the tolerance $\texttt{gtol} = 1\e-12$. 
In Table~\ref{tab:EVP_symplEV}, we present the computed symplectic eigenvalues, the $l_1$-norm errors, and the errors in the minimal value of the cost function, i.e., ${\rm err}_f=|f(X_{5000})-24|$. One can see that the optimization schemes reach the same accuracy as the MATLAB function \texttt{eigs} applied to the Hermitian pencil $A - \lambda \im J_{2n}$ with the default tolerance $1\e{-14}$.  Certainly, as less direct, these schemes are more computationally expensive. Nevertheless, an obviously seen advantage over \texttt{eigs} is that they provide purely real symplectic eigenvalues. In Figure~\ref{fig:eig_spsd_compare}, we also present  the history of the iterations consisting of the errors in the minimal value of the cost function, the norms of its Riemannian gradient, and the feasibility violation.

\begin{figure}[htbp]%
	\centering
	\includegraphics[width=0.9\textwidth]{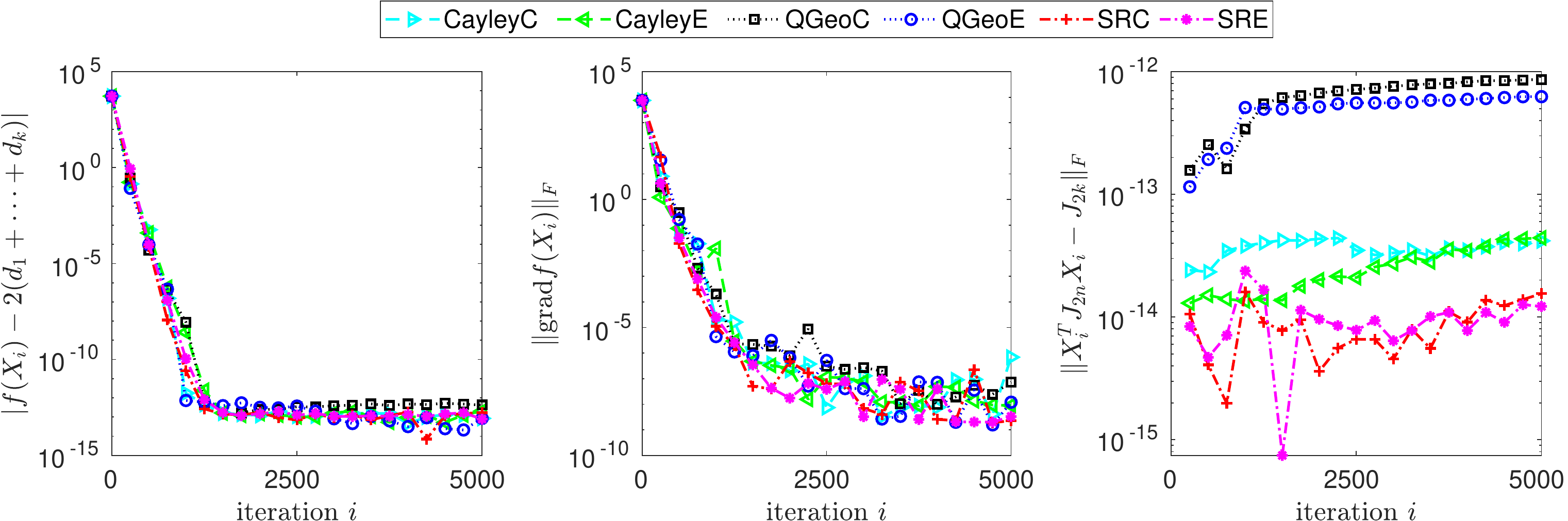}
	\caption{Symplectic eigenvalue computation for an SPSD matrix: a comparison of the cost function values (left), Riemannian gradient norms (middle), and the feasibility violation (right) for  different optimization schemes.}
	\label{fig:eig_spsd_compare}
\end{figure}

Regarding efficiency comparison, in Table~\ref{tab:EVP_symplEV_time}, we report the average time consumed in one iteration step by different optimization schemes with the aforementioned setting and with $k=n=1000$. One can see that  the schemes using the Cayley retraction are fastest in most cases. Moreover, compared to the schemes with the canonical-like metric,
those based on the Euclidean metric tend to be slightly faster  when $k\ll n$ but are obviously slower when $k$ approaches $n$.
The main reason is the difference in the orthogonal projections for computing the Riemannian gradients, c.f., Table~\ref{tab:notion_summary}: while for the canonical-like metric, one has to multiply matrices of size scalable with $2n$, 
a~Lyapunov matrix equation with coefficients of size $2k\times 2k$ is required to solve for the Euclidean metric. Therefore, for $k\ll n$,  solving such an~equation can be much faster while when $k=n$, the cost is definitely more expensive than matrix-matrix multiplication.

\begin{table}[tbp]
	\centering
	\footnotesize
	\caption{Symplectic eigenvalue computation for an SPSD matrix: 5 computed smallest symplectic eigenvalues, $l_1$-errors, and the errors in 
		the  minimal value of the cost function.
	} 
	\label{tab:EVP_symplEV}
	\begin{tabular}{crrrrrr}
		\toprule
		\multirow{2}{*}{~scheme~} & \multicolumn{3}{c}{Canonical-like (C)} & \multicolumn{3}{c}{Euclidean (E)}\\\cmidrule(r){2-4}\cmidrule(r){5-7}
		& {symplectic eigenvalues}  &   {$\;l_1$-error}\;  & $\quad{\rm err}_f\quad$ & {symplectic eigenvalues}  &   {$\;l_1$-error}\;  & $\quad{\rm err}_f\quad$   \\\midrule
		& 0.000000000000000 & & & 0.000000000000000 & & \\
		& 0.000000000000082& & & 0.000000000000049& & \\
		Cayley & 2.999999999999981& $1.42\e{-13}$ &$8.53\e{-14}$ & 3.000000000000000&$1.14\e{-13}$ &$2.06\e{-13}$\\
		& 3.999999999999988& & & 3.999999999999991&  &\\
		& 5.000000000000029& & & 5.000000000000055& &\\\cmidrule(r){2-7}
		
		&0.000000000000000 & & &0.000000000000000 & & \\
		& 0.000000000000057& & & 0.000000000000063& & \\
		{QGeo} & 2.999999999999885&$3.71\e{-13}$  & $4.23\e{-13}$ & 2.999999999999929& $1.86\e{-13}$ & $8.52\e{-14}$\\
		& 3.999999999999878& && 3.999999999999998& &\\
		& 4.999999999999924& && 5.000000000000051& &  \\\cmidrule(r){2-7}
		
		& 0.000000000000000& & & 0.000000000000000& & \\
		& 0.000000000000066& & &0.000000000000056 & & \\
		SR & 2.999999999999979& $1.06\e{-13}$ & $1.67\e{-13}$ & 2.999999999999991&$1.09\e{-13}$& $7.83\e{-14}$\\
		& 3.999999999999993& && 3.999999999999986& &\\
		& 5.000000000000012& & & 5.000000000000031& & \\\midrule[.4pt]
		
		& $\qquad(18 + 60\,\im)\e{-15}\;\,$ & & &&&\\
		& $\qquad(16 + 62\,\im)\e{-15}\;\,$& & &&&\\
		\texttt{eigs}& $3 + ((15 -\enskip 3\,\im)\e{-15})$& $2.19\e{-13}$& &&&\\
		&  $4 -((14 -11\,\im)\e{-15})$ & & &&&\\
		& $5 +((59 -\enskip 8\,\im)\e{-15})$& & &&&\\
		\bottomrule
	\end{tabular}
\end{table}

\begin{table}[tbp]
	\centering
	\footnotesize
	\caption{Symplectic eigenvalue computation for an SPSD matrix: time consumed in seconds per step by different optimization schemes.}
	\label{tab:EVP_symplEV_time}
	\begin{tabular}{c l r c c c}
		\toprule
		& & & Cayley & QGeo & SR\\ \midrule
		\multirow{4}{*}{~$n=1000$~} & \multirow{2}{*}{~$k= 5$ }&{Canonical-like (C)} & 8.50$\e{-3}$&9.38$\e{-3}$ & 9.21$\e{-3}$\\
		&& Euclidean (E) &8.75$\e{-3}$ &7.90$\e{-3}$ &7.72$\e{-3}$ \\\cmidrule(r){2-6}
		&  \multirow{2}{*}{~$k= 1000$~~} &{Canonical-like (C)}  & 2.31$\e{+0}$& 8.01$\e{+0}$ &9.88$\e{+0}$ \\
		& & Euclidean (E) &2.63$\e{+1}$ &3.64$\e{+1}$ &3.18$\e{+1}$ \\ 
		\bottomrule
	\end{tabular}
\end{table}

Regarding efficiency comparison, in Table~\ref{tab:EVP_symplEV_time}, we report the average time consumed in one iteration step by different optimization schemes with the aforementioned setting and with $k=n=1000$. One can see that  the schemes using the Cayley retraction are fastest in most cases. Moreover, compared to the schemes with the canonical-like metric,
those based on the Euclidean metric tend to be slightly faster  when $k\ll n$ but are obviously slower when $k$ approaches $n$.
The main reason is the difference in the orthogonal projections for computing the Riemannian gradients, c.f., Table~\ref{tab:notion_summary}: while for the canonical-like metric, one has to multiply matrices of size scalable with $2n$, 
a~Lyapunov matrix equation with coefficients of size $2k\times 2k$ is required to solve for the Euclidean metric. Therefore, for $k\ll n$,  solving such an~equation can be much faster while when $k=n$, the cost is definitely more expensive than matrix-matrix multiplication.

In the second example, we consider a mechanical system which is used in analysis of vibration and frequency response of wire saws \cite{WeiK00}. 
For such a~system, the SPD matrix is given by $A = J_{2n}H$ with the Hamiltonian matrix
\begin{equation*}
H = \begin{bmatrix}-\frac{1}{2}GM^{-1}&\frac{1}{4}GM^{-1}G - K\\ M^{-1}& -\frac{1}{2}M^{-1}G
\end{bmatrix},
\end{equation*}
where $M, G, K$ are the mass, damping, and stiffness matrices of the underlying mechanical system. Following the setting in \cite{SonAGS21}, we are lead to the minimization problem~\eqref{eq:tracetheorem} with $n= 2000$ and $k = 5$. As we do not know the exact symplectic eigenvalues,
in Figure~\ref{fig:eig_wiresaw_compare}, we only present the history comparison of the six optimization schemes when  Algorithm~\ref{alg:non-monotone gradient} is run with $\texttt{niter}=100000$ iterations and the tolerance $\texttt{gtol} = 1\e{-6}$.
We see that the cost function seems to stagnate even though  the Riemannian gradients are far from being zero.

\begin{figure}[htbp]	
	\centering
	\includegraphics[width=0.9\textwidth]{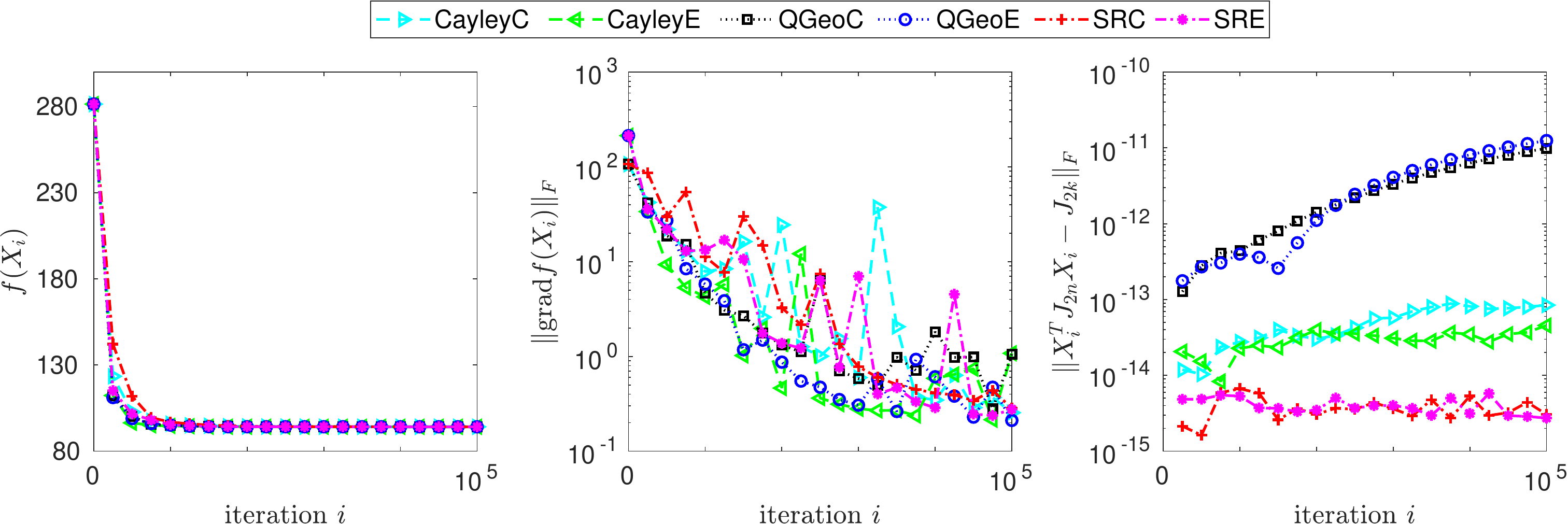}
	\caption{Symplectic eigenvalue computation for a wire saw model: a comparison of the cost function values (left), Riemannian gradient norms (middle), and the feasibility violation (right) for  different optimization schemes.}
	\label{fig:eig_wiresaw_compare}
	\end{figure}

Though with examples in this subsection, there is no considerable difference in the final values of the cost function, we can easily observe that the feasibility violation varies quite a bit for different schemes. As expected, those using the SR~retraction
apparently maintain the symplecticity constraint the best, while the schemes based on the quasi-geodesic retraction perform the worst. This is due to the fact that 
the symplecticity is retrieved by the SR~decomposition at every iteration.
This feature prevents the accumulative errors in the constraint, which can be seen in the quasi-geodesic case. This argument is actually independent of the cost function and therefore holds true for any problem. Moreover, these tests also reconfirm an observation on the constraint violation for the optimization schemes based on the Cayley and quasi-geodesic retractions made in \cite{GSAS21a}. 

\subsection{Symplectic model reduction of Hamiltonian systems}
\label{sec:smor}

The third minimization problem arises in structure-preserving model reduction of nonlinear Hamiltonian systems given by
\begin{equation}\label{eq:HS}
\dot{x} =  J_{2n}\nabla_x H(x),  \qquad x(0)=x_0, 
\end{equation} 
where 
$x\in\mathbb{R}^{2n}$ is the state vector, 
$x_0\in\mathbb{R}^{2n}$ is the initial vector,
$H:\mathbb{R}^{2n}\to\R$ is the continuously differentiable Hamiltonian function describing the internal energy of the system, $\nabla_x H(x)$ is the Euclidean gradient of~$H$ with respect to $x$, and $J_{2n}$ is the structure matrix which describes the interconnection of energy storage elements. For such systems, the Hamiltonian $H$ is a~first integral since it remains invariant along the solution of \eqref{eq:HS}. Hamiltonian systems arise in a~wide range of applications including mechanical systems, molecular dynamics, network design and electromagnetic field simulation, e.g., \cite{MarsR99,DuiMSB09,VanDSJ14}. 

The goal of model reduction is to approximate the full-order model (FOM) \eqref{eq:HS} by a~reduced-order model (ROM) which preserves the Hamiltonian structure
\begin{equation}\label{eq:redHS}
\dot{\tilde{x}}  = J_{2k}\nabla_{\tilde{x}} \tilde{H}(\tilde{x}),  \qquad \tilde{x}(0)=\tilde{x}_0,
\end{equation} 
where $\tilde{x}\in\mathbb{R}^{2k}$ is the reduced state, and $\tilde{H}:\mathbb{R}^{2k}\to\R$ is the reduced Hamiltonian with $k\ll n$. Note that 
the preservation of the Hamiltonian structure during the reduction process ensures the conservation of energy of the reduced-order model. 
In the last decade, structure- and energy-preserving model reduction of Hamiltonian systems has attracted a~lot of attention and several model reduction methods have been developed for such systems
\cite{PengM16,AfkhH17,AfkhBBH18,BuchBH19,Pag21,ShaWK21,HestPR22}.

One of the most popular model reduction approaches for nonlinear Hamiltonian systems is  the PSD method presented first in \cite{PengM16}. It is based on collecting the snapshots  $A=[x(t_1),\ldots, x(t_s)]\in\R^{2n\times s}$
of  the Hamiltonian system~\eqref{eq:HS} and determining a~projection matrix $U\in\Spkn$  which solves the following constrained  minimization problem
\begin{equation}\label{eq:min_mor}
\min_{X\in \Spkn} f(X):=\|A-XX^+A\|_F^2,
\end{equation}
where  $X^+=J_{2k}^T\,X^TJ_{2n}^{}$ is a~symplectic inverse of $X$. Although the existence of a minimizer of \eqref{eq:min_mor} is still unknown, the cost function is smooth, bounded from below, and with a monotonically decreasing line search strategy and an~experimental convergence recognized by, e.g., the distance between two consecutive iterates, one might hope that the value of the cost function at the computed point is close to the infimum.
The solution of  \eqref{eq:HS} is then approximated by $x\approx U\tilde{x}$, where $\tilde{x}\in\mathbb{R}^{2k}$ is a coordinate vector of the approximation with respect to the basis $U$. Replacing $x$ with $U\tilde{x}$ and multiplying the resulting equation from the left with $U^+$, we obtain the reduced-order system~\eqref{eq:redHS} with the reduced initial vector $\tilde{x}_0=U^+x_0$ and the reduced Hamiltonian $\tilde{H}(\tilde{x})=H(U\tilde{x})$.
It has been shown in \cite{PengM16} that the error in the Hamiltonian given by 
$\Delta H(t)=H(x(t))-\tilde{H}(\tilde{x}(t))$ 
is constant for all $t\in\mathbb{R}$. Moreover, if $x_0\in\range(U)$, then $\Delta H(t)\equiv0$. 
This implies that the reduced-order Hamiltonian system  \eqref{eq:redHS} preserves the energy. Another important property of the symplectic projection is the preservation of the stability of equilibrium points, see \cite{PengM16,AfkhH17} for details.

Due to the non-convexity and unboundedness of the feasible set $\Spkn$,  the minimization problem~\eqref{eq:min_mor}  was considered to be difficult to solve.  By imposing additional orthogonality constraint $X^TX=I$, 
different algorithms have been developed in \cite{PengM16} for computing suboptimal solutions. A~similar constraint is also used in   \cite{AfkhH17} for a~greedy algorithm applied to parametric Hamiltonian systems.  Another approach for generating a~non-orthonormal symplectic basis matrix has been presented in \cite{BuchBH19} which is based on the SVD-like decomposition \cite{Xu03} of the snapshot matrix. Unlike the standard SVD, where the optimality of the approximation is well known, the theoretical results derived in \cite{BuchBH19} does not infer a similar claim. 

Here, we employ the Riemannian optimization algorithm for  structure-preserving model reduction of Hamiltonian systems. The key potential advantage of this approach is the possibility to reach an~optimal solution instead of a suboptimal one. 
Note that since the cost function $f$ in \eqref{eq:min_mor} satisfies the homogeneity property 
\mbox{$f(XS)\!=\!f(X)$} for any $S\in\mathrm{Sp}(2k)$, one could also compute a symplectic projection matrix by solving the minimization problem on the symplectic Grassmann manifold \cite{BendZ21,BendZ22}. A~comparison of these two Riemannian optimization approaches in the context of PSD is out of the scope of this paper.

To compute symplectic reduced bases using Algorithm~\ref{alg:non-monotone gradient},
initialization is required. To reduce the risk of convergence to a saddle point as illustrated in  Subsection~\ref{ssec:Example_sympl_target}, we take the symplectic matrix $X_0$
produced by the cotangent lift method~\cite{PengM16}. It is a block diagonal orthosymplectic matrix $X_0=\diag(\hat{X},\hat{X})$, where the columns of $\hat{X}\in\mathbb{R}^{n\times k}$ are the left singular vectors of the combined snapshot matrix $\left[[I_n,\, 0]A, \;  [0,\, I_n]A\right]$. As we start with a suboptimal solution, it is expected that the trial step size $\gamma_i$  should generally be small, and therefore,  we set  $\gamma_0 = 1\e{-8}$.
As will be shown below, the optimization method always improves this result and yields smaller approximation errors in model reduction.

We also compare our optimization-based model reduction methods with other existing reduction techniques for Hamiltonian systems. Among the PSD approaches proposed in~\cite{PengM16}, we choose the cotangent lift method due to its good performance, where the balance between the simplicity and the accuracy is taken into account. Moreover, whenever possible, we include the PSD SVD-like decomposition method from \cite{BuchBH19} in our comparison experiments.

As we are specially interested in the energy conservation
in the course of model reduction of Hamiltonian systems, we employ the Crank--Nicolson integration method with a~constant time step size $h_t$. 
It has been shown in~\cite{LiWS16} that this method 
is of second order and that it delivers a discrete solution whose mass and energy are conserved.
The Crank--Nicolson scheme, being implicit, requires  the numerical solution of nonlinear systems in each time step. For this purpose, we use  in our experiments the MATLAB function \texttt{fsolve}. 

For nonlinear dynamical systems, one has to additionally approximate the nonlinear term in order to maintain the benefit of smaller order of the ROM~\eqref{eq:redHS}. This task for nonlinear Hamiltonian systems is even more challenging since the approximation must be done in such a way that it does not destroy the Hamiltonian structure which guarantees the energy conservation. For this purpose, different approaches can be used. 
\begin{itemize}
	\item[1)]  The first one is the PSD-DEIM (also termed as SDEIM in \cite{PengM16}), which is a~combination of the PSD proposed in \cite{PengM16} and the DEIM developed in \cite{ChaS10}. 	Assuming that the Hamiltonian function in \eqref{eq:HS} has the form
	\begin{equation*} 
	H(x)=\frac{1}{2}x^TMx+h(x)
	\end{equation*}
	with an SPD  matrix $M\in\mathbb{R}^{2n\times 2n}$ and a~nonlinear function $h:\mathbb{R}^{2n}\to\mathbb{R}$, the gradient $\nabla_{\tilde{x}}\tilde{H}(\tilde{x})$ in the ROM \eqref{eq:redHS} is then approximated by
	\begin{equation}\label{eq:SDEIM_scheme}
	\nabla_{\tilde{x}}\tilde{H}(\tilde{x})\approx 
	U^TMU\tilde{x}+U^TV(P^TV)^{-1}P^T \nabla_x h(U\tilde{x}),
	\end{equation}
	where $V\in\mathbb{R}^{2n\times m}$ is a~DEIM basis matrix 
	and $P=[e_{i_1},\ldots,e_{i_m}]$ is a~selector matrix associated with an index set
	$\{i_1,\ldots,i_m\}$ determined by a~greedy procedure applied to~$V$.
	\item[2)] The second approach is the structure-preserving method developed for port-Hamiltonian systems in~\cite{ChatBG16}. In this case, one uses the approximation
	\begin{equation}\label{eq:Structure_Preserving_scheme}
	\nabla_{\tilde{x}}\tilde{H}(\tilde{x})\approx 
	U^TMU\tilde{x}+U^TV(P^TV)^{-1}P^T \nabla_x h(P(V^TP)^{-1}V^TU\tilde{x}).
	\end{equation}
	The resulting model reduction  method is referred to as structure-pre\-ser\-ving PSD-DEIM.
\end{itemize}
The key point of the approximations \eqref{eq:SDEIM_scheme} and \eqref{eq:Structure_Preserving_scheme} is that only a~small number of selected components of the nonlinear term are needed to be evaluated. In our implementation, this number $m$ is approximately set to~$2.5 k$.
One can observe that the PSD-DEIM approach~\eqref{eq:SDEIM_scheme} focuses more on the task of approximation and actually the resulting ROM deviates from being a~Hamiltonian system. Nevertheless, the rate of this deviation
is shown to be bounded and this upper bound depends on the approximation quality of the nonlinear term \cite[Theorem 5.1]{PengM16}. On the contrary, the approximation~\eqref{eq:Structure_Preserving_scheme}  seems to devote more to the  structure preservation task. Indeed, it provides a~reduced-order Hamiltonian model with an approximate Hamiltonian function 
\[
\tilde{H}(\tilde{x})=\frac{1}{2}\tilde{x}^TU^TMU\tilde{x}+h(P(V^TP)^{-1}V^TU\tilde{x}).
\] 
In our opinion, however, the approximation $h(P(V^TP)^{-1}V^TU\tilde{x})$
can unexpectedly cause large  errors in the Hamiltonian function and the state.
In Subsection~\ref{subsec:mor_tests}, we present 
a comparison of these two approaches. In our experiments, for model reduction of nonlinear-systems, we use the PSD-DEIM method unless stated otherwise.  

To verify the approximation properties of ROMs, we consider the relative errors in the state vector and the energy given by	 
\begin{equation}\label{eq:RE_xH}
{\rm RE}_x=\dfrac{\|x-\tilde{x}\|_{L^2(0,T)}}{\|x\|_{L^2(0,T)}}, \qquad\quad {\rm RE}_H=\dfrac{\|H(x(\cdot))-\tilde{H}(\tilde{x}(\cdot))\|_{L^2(0,T)}}{\|H(x(\cdot))\|_{L^2(0,T)}},
\end{equation}
respectively, where $\|\cdot\|_{L^2(0,T)}$ denotes the $L^2$-norm of the corresponding function. Moreover, for each model, we choose one ROM of certain dimension and present the relative state vector error $\|x(t) - \tilde{x}(t)\|/\mu_{[0,T]}(\|x\|)$, where $\mu_{[0,T]}(\|x\|)$ denotes the mean of the function $\|x(t)\|$ over the time interval $[0,T]$, 
and the relative energy error $|H(x(t))-\tilde{H}(\tilde{x}(t))|/|H(x(0))|$ versus the time. We also expose the so-called \emph{average accelerating factor} (a.a.f.) of MOR which is the ratio of the simulation time for the FOM~\eqref{eq:HS} over the average of those for the ROMs~\eqref{eq:redHS} computed by different model reduction methods.

In the following, we present the results of various numerical experiments. In each example, we first introduce a model in the form of PDEs and its associated energy. This model is then spatially discretized using an~appropriate method to give a~finite-dimensional Hamiltonian system of the form~\eqref{eq:HS} with the  corresponding semi-discrete energy. Furthermore, physical and numerical parameters are set up.

In the preliminary tests, we have compared different optimization schemes applied to all models similarly to Subsection~\ref{ssec:symplEV}. 
Since the symplecticity constraint is important in preserving the Hamiltonian structure, we choose the ones that maintain this constraint better. Moreover, among those, we also want to see the possible difference resulting from the use of different metrics. Therefore, for simplicity, we restrict our comparisons in Subsections~\ref{subsec:NumerExampl_wave}--\ref{subsec:NumerExampl_Vlasov} to the CayleyC and SRE schemes. In addition, general discussion 
is given in Subsection~\ref{subsec:mor_tests}.

\subsubsection{Linear wave equation}
\label{subsec:NumerExampl_wave}

As a~first example, we consider a~1D linear wave equation subject to the periodic boundary   conditions 
\begin{equation}\label{eq:WaveEq}
\arraycolsep=2pt
\begin{array}{rclcl}
z_{tt} & = & c^2\, z_{\xi\xi} &\quad \text{ in } & (0,T)\times (a,b),\\
z(t,a) & = & z(t,b)        &\quad \text{ in } & (0,T),\\
z(0,\xi)& = & z_0(\xi),\enskip z_t(0,\xi) = z_1(\xi) &\quad \text{ in } & (a,b),
\end{array}
\end{equation}
where constant $c>0$ is the wave speed and $z$ is the unknown function of time and space, see \cite{PengM16}. The Hamiltonian function associated with \eqref{eq:WaveEq} is given by
\begin{equation}\label{eq:Ham_lin_wave}
\mathcal{H}(z)=\int_a^b \frac{1}{2} z_t^2 +\frac{c^2}{2} z_\xi^2 \> {\rm d}\xi.
\end{equation}

Using a finite difference discretization method on a~uniform spatial grid $\xi_j = a+j h_\xi$, $j = 0,...,n$, with the spatial step size $h_\xi=(b-a)/n$ and setting 
$q_j(t) = z(t,\xi_j)$, \mbox{$p_j(t) = z_t(t,\xi_j)$}, $q = [q_1,\ldots, q_n]^T$, and
$p = [p_1, \ldots, p_n]^T$, we obtain the linear Hamiltonian system
\begin{align}\label{eq:WaveHamiltonSyst}
\begin{bmatrix}
\dot{q} \\ \dot{p}
\end{bmatrix} &= J_{2n}\begin{bmatrix}
-c^2 D_{\xi\xi}&0\\0& I
\end{bmatrix} \begin{bmatrix}
q\\p
\end{bmatrix} 
\end{align}
with the quadratic Hamiltonian function 
\[
H(q,p) = \sum_{j=1}^{n} \frac{p_j^2}{2} + \frac{c^2(q_{j+1}-q_j)^2}{2h_\xi^2} + \frac{c^2(q_j-q_{j-1})^2}{2h_\xi^2} 
=\frac{1}{2}p^Tp-\frac{c^2}{2}q^TD_{\xi\xi}\,q,
\]
where $D_{\xi\xi}$ is the three-point finite difference matrix approximating the second-order spatial differential operator. Note that $h_\xi H(q,p)$ gives the spatial discretization of the Hamiltonian function~\eqref{eq:Ham_lin_wave}.
In our tests, we chose $c=0.1$, $a=0$, $b=1$, $h_\xi=0.002$, $T=50$, $h_t=0.01$, and the initial functions $z_1(\xi)\equiv 0$ and $z_0(\xi)=\phi(10|\xi-\tfrac12|)$ with the cubic spline function
\[
\phi(\eta)=\left\{\begin{array}{lcl}
1-\frac{3}{2}\eta^2+\frac{3}{4}\eta^3 & \text{ if } & 0\leq \eta\leq 1, \\[1mm]
\frac{1}{4}(2-\eta)^3 & \text{ if } & 1< \eta\leq 2, \\[1mm]
0           & \text{ if } & 2 < \eta.
\end{array}\right. 
\]

\noindent
The resulting system \eqref{eq:WaveHamiltonSyst} has dimension 
$2n=1000$. For the PSD data, we extract $s=500$ snapshots from the Crank--Nicolson solution of the FOM and compute the ROMs of dimension $2k$ with $k = 10, 20, 40, 80$. The relative errors defined in~\eqref{eq:RE_xH} and a.a.f.
for the ROMs computed by the cotangent lift (CotLift), the Riemannian optimization-based  methods CayleyC and  SRE with $\texttt{iter}=1000$ iterations,
and the SVD-like reduction method (SVD-like) are presented in Table~\ref{tab:LinWaveEq}. 
Furthermore, Figure~\ref{fig:LinWaveEq} shows the changes of the relative state and energy errors with time for the ROM of dimension $2k=80$. 

\begin{table}[htbp]
	\centering
	\footnotesize
	\caption{Linear wave equation, $n=500$: model reduction results
		(the italic font means a~warning of ill-conditioning received during the simulation of the reduced-order model).}
	\label{tab:LinWaveEq}
	\begin{tabular}{r r r r r r r r r r }
		\toprule
		\multirow{2}{*}{~$k$~} & \multicolumn{2}{c}{CotLift} & \multicolumn{2}{c}{CayleyC} & \multicolumn{2}{c}{SRE}&\multicolumn{2}{c}{SVD-like} & \multirow{2}{*}{a.a.f.}\\\cmidrule(r){2-3}\cmidrule(r){4-5}\cmidrule(r){6-7}\cmidrule(r){8-9}
		& $\mathrm{RE}_x$ & $\mathrm{RE}_H$ & $\mathrm{RE}_x$ & $\mathrm{RE}_H$ & $\mathrm{RE}_x$ & $\mathrm{RE}_H$ & $\mathrm{RE}_x$ & $\mathrm{RE}_H$ &  \\\midrule
		10 & 4.91$\e{-2}$ & $1.45\e{-12}$ &4.90$\e{-2}$ & $1.46\e{-12}$ &4.90$\e{-2}$ &$1.12\e{-12}$&1.78$\e{-1}$ &$1.28\e{-2}$&73 \\
		20&1.63$\e{-2}$ &$2.26\e{-12}$ &1.61$\e{-2}$ &$2.74\e{-12}$ &1.60$\e{-2}$ & $9.88\e{-13}$&7.07$\e{-2}$& $1.54\e{-4}$& 48\\
		40& 9.79$\e{-3}$& $4.22\e{-12}$&9.15$\e{-3}$ & $7.32\e{-12}$&9.05$\e{-3}$ &$1.10\e{-12}$&\emph{5.71e$-$2}&\emph{1.23e$-$4}&36 \\
		80&6.28$\e{-3}$ &$1.56\e{-12}$ &6.23$\e{-3}$ &$1.23\e{-11}$ &6.23$\e{-3}$ &$2.86\e{-12}$&7.92$\e{-2}$&$1.96\e{-5}$&\ \,3 \\
		\bottomrule
	\end{tabular}
\end{table}

\begin{figure}[htbp]%
	\centering
	\includegraphics[width=0.7\textwidth]{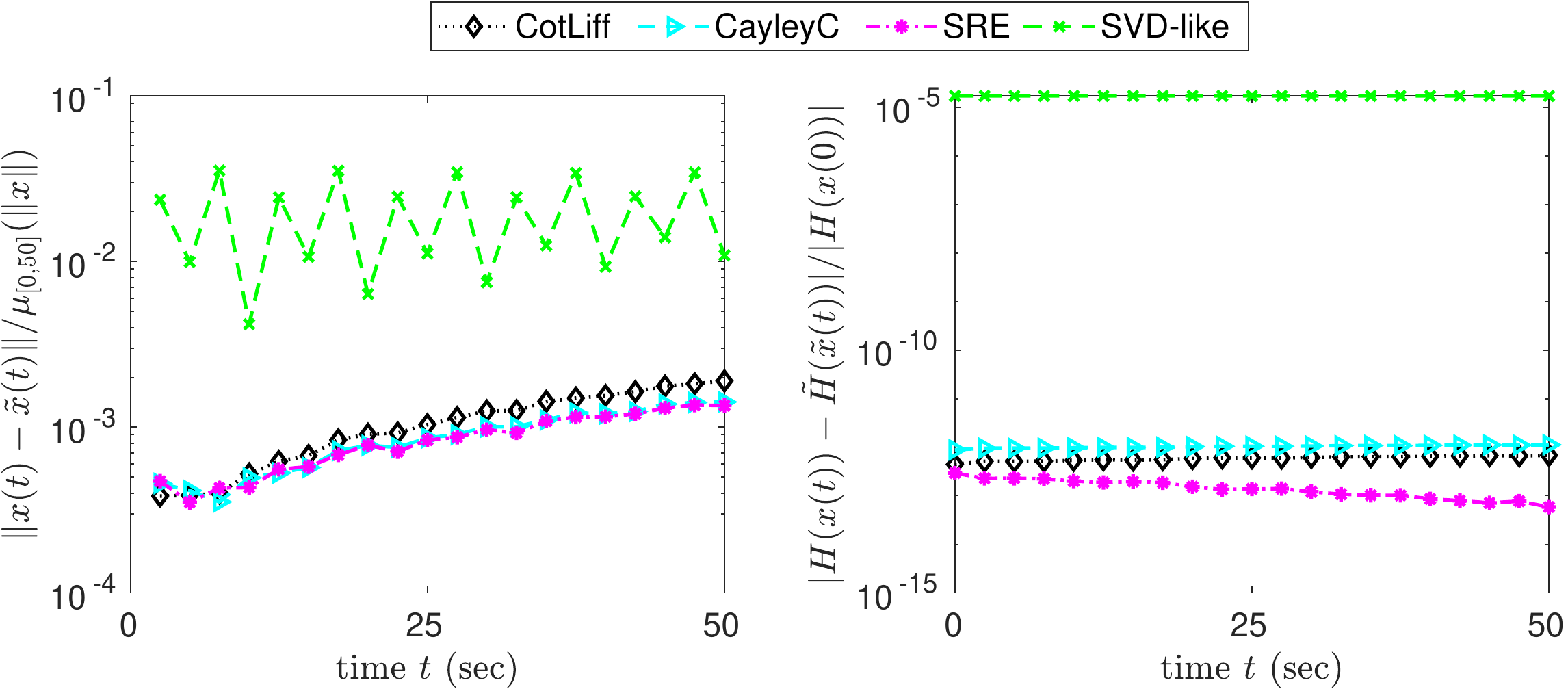}
	{\caption{Linear wave equation, $n=500$, $k=40$: relative errors in the state vector (left) and the Hamiltonian function (right).}
		\label{fig:LinWaveEq}}
\end{figure}

\subsubsection{Sine-Gordon equation}
The second example is the sine-Gordon equation with Dirichlet boundary conditions
\begin{equation}\label{eq:sineGordon}
\arraycolsep=2pt
\begin{array}{rclcl}
z_{tt} & = & z_{\xi\xi}-\sin(z) &\quad \text{ in } & (0,T)\times (a,b),\\
z(t,a) & = & \phi_a(t), \quad z(t,b)=\phi_b(t),    &\quad \text{ in } & (0,T),\\
z(0,\xi)& = & z_0(\xi),\enskip z_t(0,\xi) = z_1(\xi) &\quad \text{ in } & (a,b).
\end{array}
\end{equation}
Such an~equation appears, for example, in differential geometry and in a~wide variety of
physical applications. For the initial conditions
\[
z_0(\xi) = 4\arctan \biggl(\exp\Bigl(\frac{\xi-\xi_0}{\sqrt{1-v^2}}\Bigr)\biggr),
\qquad z_1(\xi) = \frac{-4\, v\,\exp\bigl(\frac{\xi-\xi_0}{\sqrt{1-v^2}}\bigr)}
{\sqrt{1-v^2}\Bigl(1+\exp\bigl(\frac{2(\xi-\xi_0)}{\sqrt{1-v^2}}\bigr)\Bigr)} ,
\]
the sine-Gordon equation \eqref{eq:sineGordon} admits the exact solitary solution
\[
z(t,\xi)=4\arctan\biggl(\exp\Bigl(\frac{\xi-\xi_0-vt}{\sqrt{1-v^2}}\Bigr)\biggr),
\] 
where $0<v<1$ is the velocity of the solitary wave and $\xi_0>0$.  

The finite difference method on a~uniform spatial grid $\xi_j = j h_\xi$ for  
\mbox{$j = 0,...,n+1$}, with the spatial step size $h_\xi=(b-a)/(n+1)$  
leads to the nonlinear Hamiltonian system
\begin{equation}\label{eq:sineGrodonHamilton}
\begin{bmatrix}
\dot{q}\\ \dot{p}
\end{bmatrix} = J_{2n}\left(\begin{bmatrix}
-D_{\xi\xi}&0\\0& I
\end{bmatrix} \begin{bmatrix}
q\\p
\end{bmatrix} 
+ \begin{bmatrix}\psi(q)\\0\end{bmatrix}\right),
\end{equation}
where $D_{\xi\xi}$ is the three-point finite difference matrix, $q = [q_1,\ldots, q_n]^T$ with $q_j(t) = z(t,\xi_j)$, 
$p = [p_1, \ldots, p_n]^T$ with $p_j(t) = z_t(t,\xi_j)$, and
\[
\psi(q)=\bigl[\sin(q_1)-\frac{\phi_a}{h_\xi^2},\sin(q_2),\ldots,\sin(q_{n-1}),\sin(q_n)-\frac{\phi_b}{h_{\xi}^2}\bigr]^T.
\]
Note that $D_{\xi\xi}$ in \eqref{eq:sineGrodonHamilton} slightly differs from that in the semi-discretized wave equation~\eqref{eq:WaveHamiltonSyst} due to the different boundary conditions. The Hamiltonian function of \eqref{eq:sineGrodonHamilton} is given by
\[
H(q,p)=\frac{1}{2}p^Tp -\frac{1}{2}q^TD_{\xi\xi}q   
+ \sum_{i=1}^{n} \bigl(1\!-\!\cos(q_i)\bigr)+
\frac{\phi_a^2}{2h_{\xi}^2}-\frac{\phi_aq_1}{h_{\xi}^2}+\frac{\phi_b^2}{2h_{\xi}^2}-
\frac{\phi_bq_n}{h_{\xi}^2}.
\]
In our experiments, we take $v=0.2$, $\xi_0=10$, $a=0$, $b=50$, $h_{\xi}=0.025$, \mbox{$n = 1001$}, $T=90$,  and $h_t=0.05$. After the FOM simulation, $s=450$ snapshots are uniformly extracted from the solution for constructing the reduced bases of size $2k$ with $k = 11, 13, 15, 17$ by using the proposed optimization schemes with $\texttt{iter}=6000$ iterations.
Model reduction results with detailed setting values are given in Table~\ref{tab:SineGordon} and Figure~\ref{fig:SineGordon}.

\begin{table}[htbp]
	\caption{Sine-Gordon equation, $n=1001$: model reduction results.}
	\label{tab:SineGordon}
	\centering
	\footnotesize
	\begin{tabular}{r r r r r r r r }
		\toprule
		\multirow{2}{*}{~$k$~} & \multicolumn{2}{c}{CotLift} & \multicolumn{2}{c}{CayleyC} & \multicolumn{2}{c}{SRE}& \multirow{2}{*}{a.a.f.}\\\cmidrule(r){2-3}\cmidrule(r){4-5}\cmidrule(r){6-7}
		& $\mathrm{RE}_x$ & $\mathrm{RE}_H$ & $\mathrm{RE}_x$ & $\mathrm{RE}_H$ & $\mathrm{RE}_x$ & $\mathrm{RE}_H$ &  \\\midrule
		11&$1.43\e{-2}$ &$4.11\e{-2}$ &$1.70\e{-2}$ & $9.70\e{+0}$ &$9.38\e{-3}$ & $1.37\e{-2}$ &118\\
		13& $8.32\e{-3}$& $1.69\e{-2}$&$4.21\e{-3}$ & $3.97\e{-3}$ &$6.78\e{-3}$ &$3.72\e{-3}$&100\\
		15&$5.65\e{-3}$ &$9.12\e{-3}$ &$3.60\e{-3}$ &$7.41\e{-4}$ &$4.97\e{-3}$ &$2.56\e{-3}$& \ \,92\\
		17 &$3.34\e{-3}$&$4.57\e{-3}$ &$3.05\e{-3}$&$1.00\e{-3}$&$3.43\e{-3}$&$1.62\e{-3}$&\ \,86\\
		\bottomrule
	\end{tabular}
\end{table} 
\begin{figure}[htbp]%
	\centering
	\includegraphics[width=0.7\textwidth]{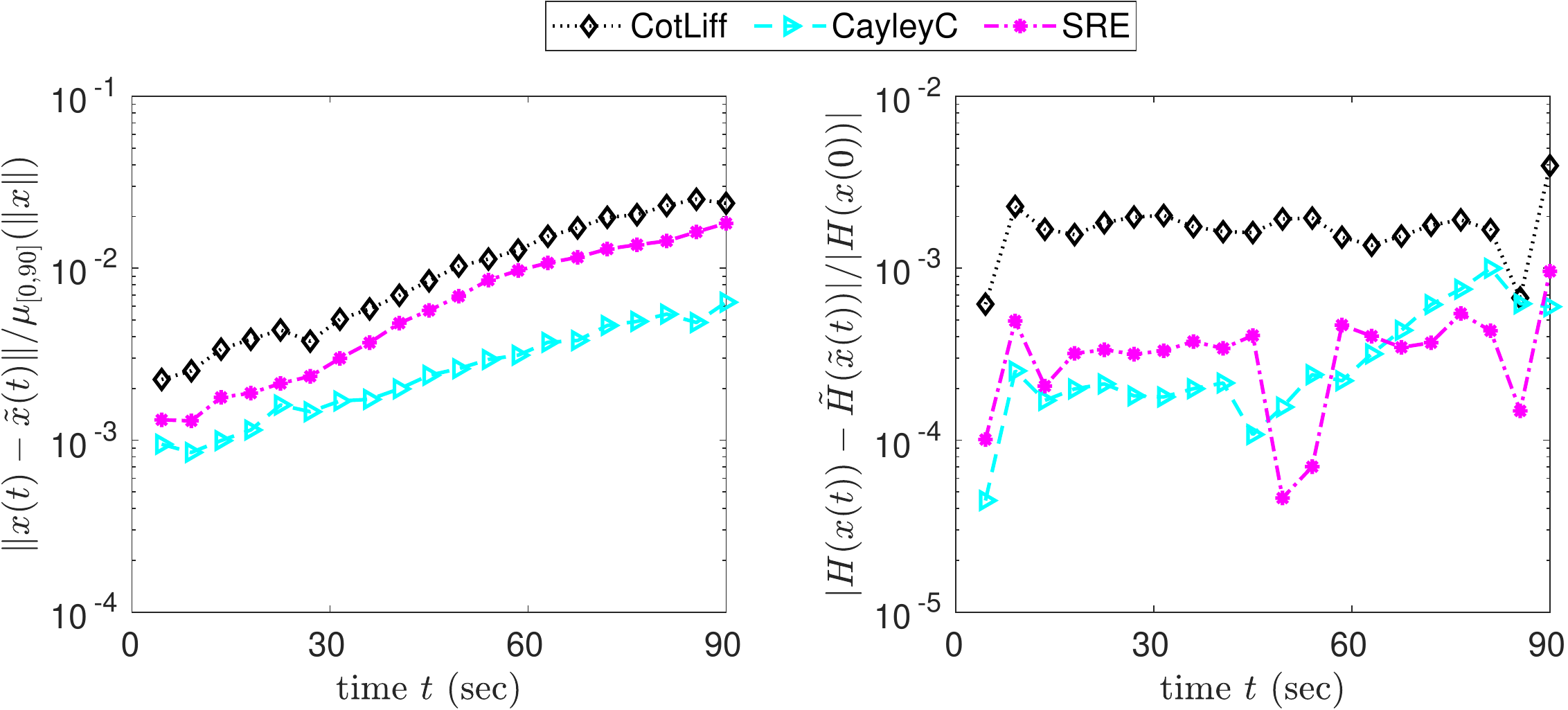}
	{\caption{Sine-Gordon equation, $n = 1001, k=13$: relative errors in the state vector (left) and the Hamiltonian function (right).}
		\label{fig:SineGordon}}
\end{figure}

\subsubsection{Schr\"{o}dinger equation}
\label{subsec:NumerExampl_Schroedinger}

The third test is a 1D nonlinear Schr\"{o}dinger equation with periodic boundary conditions 
\begin{equation}\label{eq:SchroedingerEq}
\arraycolsep=2pt
\begin{array}{rclcl}
\mathrm{i}\,z_t &=& -z_{\xi\xi}- \varepsilon|z|^2z &\quad \text{ in }& (0,T)\times (-\frac{L}{2},\frac{L}{2}),\\
z(t,-\frac{L}{2}) & = & z(t,\frac{L}{2}) & \quad \text{ in }& (0,T),\\
z(0,\xi) &= & z_0(\xi)& \quad \text{ in }& (-\frac{L}{2},\frac{L}{2}),
\end{array}
\end{equation}
where $z$ is the unknown complex valued wave function, $\varepsilon>0$, 
and the initial condition has the form
\[
z_0(\xi) = \frac{\sqrt{2}}{\cosh(\xi-\xi_0)}\exp\Bigl(\mathrm{i}\frac{c\,(\xi-\xi_0)}{2}\Bigr)
\]
with the wave speed $c>0$ and $\xi_0\geq 0$. By introducing the real and imaginary part of the wave function $z = z_1 + \mathrm{i}\,z_2$,
we can turn the Schr\"{o}dinger equation~\eqref{eq:SchroedingerEq} into an~infinite-dimensional Hamiltonian system 
\begin{align}\label{eq:SchroedingerHamiltonform1}
\begin{split}
(z_1)_t &=         -(z_2)_{\xi\xi} - \varepsilon(z_1^2+z_2^2)z_2,\\
(z_2)_t &= \enskip\,\,(z_1)_{\xi\xi} + \varepsilon(z_1^2+z_2^2)z_1,
\end{split}
\end{align}
with the associated Hamiltonian function
\[
\mathcal{H}(z_1,z_2) = \int_{0}^{L} \frac{1}{2}(z_1)_\xi^2 + \frac{1}{2}(z_2)_\xi^2 - \frac{\varepsilon}{4}(z_1^2 + z_2^2)^2\, {\rm d}\xi.
\]

A~spatial discretization of \eqref{eq:SchroedingerHamiltonform1} by using the finite difference method on a~uniform spatial grid $\xi_j = jh_\xi$ with $h_\xi=L/n$ leads to the nonlinear Hamiltonian system
\begin{align}\label{eq:SchroedingerHamiltonform2}
\begin{split}
\begin{bmatrix}
\dot{q}\\\dot{p}
\end{bmatrix} &= J_{2n}\begin{bmatrix}
- D_{\xi\xi}-G(q,p) &0\\0& -D_{\xi\xi}-G(q,p)
\end{bmatrix} \begin{bmatrix}
q\\p
\end{bmatrix},
\end{split}
\end{align}
where $q = [q_1,\ldots,q_n]^T$ with $q_j(t) = z_1(t,\xi_j)$,  
$p = [p_1,\ldots,p_n]^T$ with $p_j(t) = z_2(t,\xi_j)$, 
$D_{\xi\xi}$ is  the three-point finite difference matrix,
and 
$$
G(q,p)=\varepsilon\diag(q_1^2 + p_1^2,\ldots, 	q_n^2 + p_n^2).
$$ 
The Hamiltonian function of \eqref{eq:SchroedingerHamiltonform2} has the form
\[
H(q,p) = \sum_{i = 1}^{n}\biggl(\frac{q_i^2-q_iq_{i-1}}{h_\xi^2} + \frac{p_i^2-p_ip_{i-1}}{h_\xi^2} -\frac{\varepsilon}{4}(p_i^2+q_i^2)^2\biggr).
\] 
Numerical setting for simulation is similar to that in \cite{AfkhH17}:
$\varepsilon=1.0932$, \mbox{$L = 2\pi/0.11$}, $c=1$, $\xi_0 = 0$, $h_\xi = 0.2231$, $n = 1024$, $T=30$, and $h_t=0.01$.  From the FOM solution at 3000 time instances, we employ $s=750$ of them for computing the ROMs of dimension~$2k$ with $k = 95, 100, 105, 110$. We run the optimization schemes with ${\tt iter}=500$ iterations. Similarly, we report the model reduction results in Table~\ref{tab:Schroedinger} and Figure~\ref{fig:Schroedinger}. 

\begin{table}[ht]
	\caption{Schr\"{o}dinger equation, $n=1024$: model reduction results.}
	\label{tab:Schroedinger}
	\centering
	\footnotesize
	\begin{tabular}{r r r r r r r r }
		\toprule
		\multirow{2}{*}{~$k$~} & \multicolumn{2}{c}{CotLift} & \multicolumn{2}{c}{CayleyC} & \multicolumn{2}{c}{SRE}& \multirow{2}{*}{a.a.f.}\\\cmidrule(r){2-3}\cmidrule(r){4-5}\cmidrule(r){6-7}
		& $\mathrm{RE}_x$ & $\mathrm{RE}_H$ & $\mathrm{RE}_x$ & $\mathrm{RE}_H$ & $\mathrm{RE}_x$ & $\mathrm{RE}_H$ &  \\\midrule
		95 & $7.62\e{-2}$ & $1.22\e{-1}$ &$8.06\e{-2}$ &$1.35\e{-1}$ &$5.88\e{-2}$ &$7.62\e{-2}$&1.74 \\
		100 & $6.77\e{-2}$ & $9.67\e{-2}$ &$1.98\e{-2}$ &$8.46\e{-3}$ &$2.94\e{-2}$ &$1.85\e{-2}$&1.52 \\
		105&$4.48\e{-2}$ &$4.30\e{-2}$ & $4.71\e{-2}$ &$4.86\e{-2}$ &$2.14\e{-2}$ & $9.77\e{-3}$&1.36\\
		110& $3.60\e{-2}$& $2.78\e{-2}$&$3.95\e{-2}$ & $3.41\e{-2}$ & $3.42\e{-2}$ & $2.55\e{-2}$ &1.35\\
		\bottomrule
	\end{tabular}
\end{table}

\begin{figure}[ht]%
	\centering
	\includegraphics[width=0.7\textwidth]{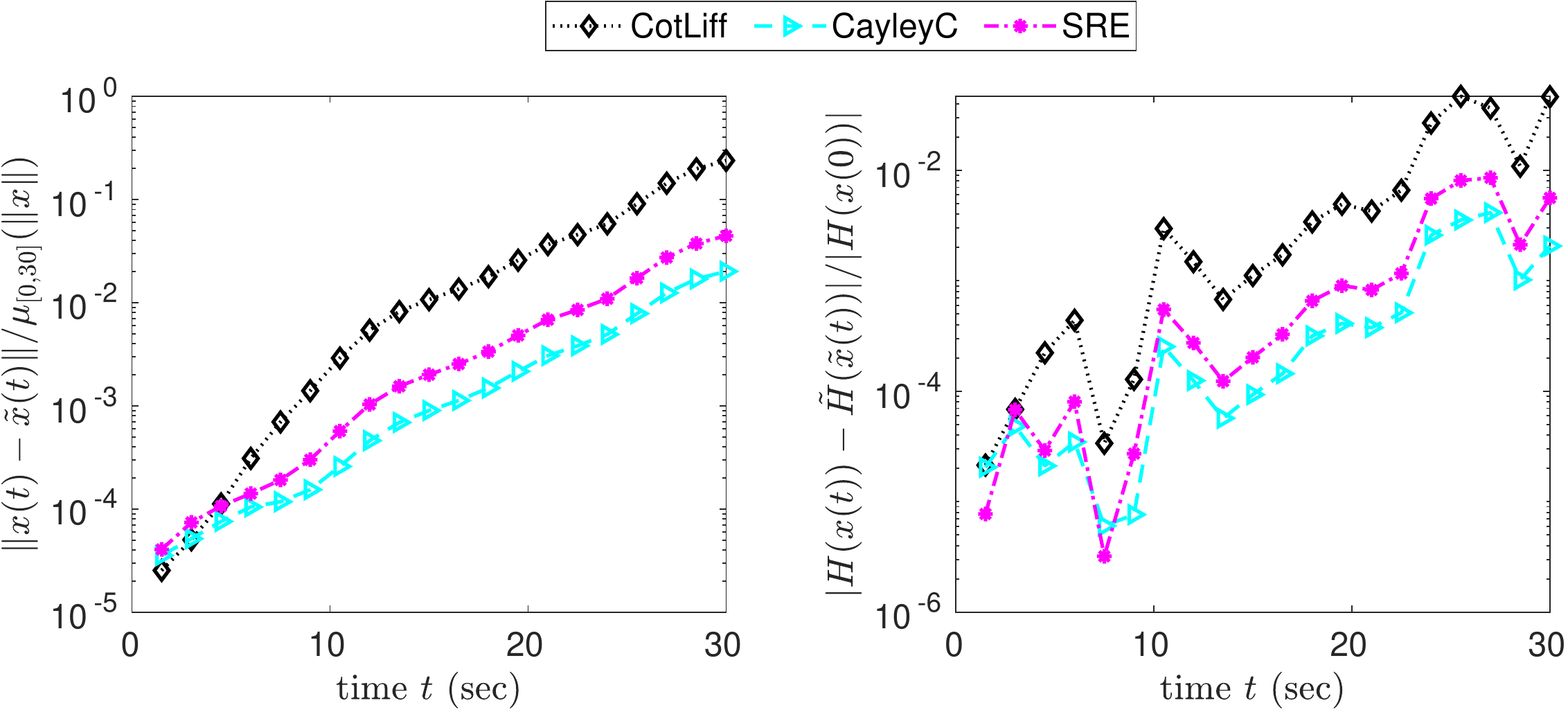}
	{\caption{Schr\"{o}dinger equation, $n=1024$, $k=100$:
			relative errors in the state vector (left) and the Hamiltonian function (right).}
		\label{fig:Schroedinger}}
\end{figure}

\subsubsection{Vlasov equation}
\label{subsec:NumerExampl_Vlasov}

Finally, we consider a~1D Vlasov equation \cite{TyraK19}
\begin{equation}\label{eq:VlasovEq}
z_t + v^T z_{\xi}  -E(\xi)^T  z_v= 0,
\end{equation}
where  $z = z(t,\xi,v)$ is the particle density function, 
$E(\xi) = -\tfrac{\rm d}{{\rm d}\xi}\phi(\xi)$ is the electric field with a~potential function $\phi(\xi)$. 
In the particle-in-cell method~\cite{Lape16}, one assumes that there are totally $n$ particles whose position and velocity at time $t$ are denoted, respectively, by $q_j(t)$ and $p_j(t)$, $j = 1,\ldots,n$. Approximating~$z$~by 
\[
\tilde{z}(t,\xi,v)  = \sum_{j=1}^{n}w_j\delta(\xi-q_j(t))\delta(v-p_j(t)),
\]
where $\delta$ is the Dirac function and $w_j$ is the $j$-th weight, and 
requiring that the zeroth and first-order moments in $\xi$ and $v$ of \eqref{eq:VlasovEq} are satisfied for this approximation, we obtain the equations
\begin{equation}\label{eq:XVEq}
\dot{q}_j = p_j,\quad \dot{p}_j = -E(q_j),\quad j = 1,\dots,n.
\end{equation}
Introducing $q = [q_1^T,\ldots,q_n^T]^T$ and $p = [p_1^T,\ldots, p_n^T]^T$,
equations \eqref{eq:XVEq} can shortly be written as the Hamiltonian system \eqref{eq:HS} with $x=[q^T,\,p^T]^T$ and the Hamiltonian
function
\[
H(q,p) = \sum_{j=1}^{n}\biggl(\frac{1}{2}p_i^Tp_i^{} - \phi(q_i)\biggr).
\]

We use the same numerical setting as in \cite{TyraK19}: $t \in (0,0.2]$, $\xi \in [0, 1]$, the periodic boundary condition and the initial condition
\begin{equation}\label{eq:ICf}
z(0,\xi,v) = \frac{1+\varepsilon\cos (2\pi \xi)}{\sqrt{2\pi}(a+1)}\left(\exp\Big(\!-\frac{v^2}{2}\Big) +\frac{a}{\sigma}\,\exp\Big(\!-\frac{(v-v_0)^2}{2\sigma^2}\Big) \right)
\end{equation}
with $\varepsilon = 0.3$, $a = 0.3$, $v_0 = 4$, and $\sigma = 1$. 
We assume that there are $n\!=\!1000$ particles in the system. The time interval is equally divided into subintervals of length $h_t= 0.0001$. The initial values $q(0)$ and $p(0)$ are generated as random variables from the distribution\footnote{Tristan Ursell (2020). Generate Random Numbers from a 2D Discrete Distribution, MATLAB Central File Exchange. Retrieved April 29, 2020. (\href{https://www.mathworks.com/matlabcentral/fileexchange/35797-generate-random-numbers-from-a-2d-discrete-distribution}{https://www.mathworks.com/matlabcentral/fileexchange/35797-generate-random-numbers-from-a-2d-discrete-distribution})}~\eqref{eq:ICf}. The electric field is chosen as $E(\xi) = 3\cos (4\pi \xi)$. Among $2000$ time  instances of the FOM solution, $s=400$ snapshots are used for constructing the ROMs of dimension $2k$ with $k=6, 8, 10, 12$. In the optimization schemes, we choose ${\tt iter}=1000$.
Table~\ref{tab:Vlasov} and Figure~\ref{fig:Vlasov} show the obtained model reduction results.

\begin{table}[ht]
	\caption{Vlasov equation, $n=1001$: model reduction results.}
	\label{tab:Vlasov}
	\centering
\footnotesize
	\begin{tabular}{r r r r r r r r r r }
		\toprule
		\multirow{2}{*}{~$k$~} & \multicolumn{2}{c}{CotLift} & \multicolumn{2}{c}{CayleyC} & \multicolumn{2}{c}{SRE}&\multicolumn{2}{c}{SVD-like} & \multirow{2}{*}{a.a.f.}\\\cmidrule(r){2-3}\cmidrule(r){4-5}\cmidrule(r){6-7}\cmidrule(r){8-9}
		& $\mathrm{RE}_x$ & $\mathrm{RE}_H$ & $\mathrm{RE}_x$ & $\mathrm{RE}_H$ & $\mathrm{RE}_x$ & $\mathrm{RE}_H$ & $\mathrm{RE}_x$ & $\mathrm{RE}_H$ &  \\\midrule
		$6$ & $1.76\e{-3}$ & $1.32\e{-\,\ 6}$ &$1.65\e{-3}$ &$2.40\e{-\ \,7}$ &$1.10\e{-3}$ &$6.67\e{-\,\ 7}$& $2.94\e{-4}$&$4.37\e{-\ \,9}$ &$13$\\
		8&$5.04\e{-4}$ &$6.53\e{-\,\ 9}$ &$4.75\e{-4}$ &$1.94\e{-\,\ 8}$ &$5.11\e{-4}$ & $5.11\e{-\,\ 9}$ &$7.61\e{-5}$& $7.56\e{-11}$&$12$\\
		10& $2.23\e{-4}$ & $4.68\e{-10}$& $2.22\e{-4}$ & $6.73\e{-10}$&$1.19\e{-4}$ &$1.08\e{-\ \,9}$ & $5.43\e{-4}$ & $6.00\e{-\ \,9}$&$11$\\
		12&$6.32\e{-5}$&$5.33\e{-11}$ &$6.30\e{-5}$ &$5.26\e{-11}$ &$6.31\e{-5}$ &$5.20\e{-11}$ & $2.05\e{-2}$ &$1.46\e{-\,\ 3}$&$10$\\
		\bottomrule
	\end{tabular}
\end{table}
\begin{figure}[t]%
	\centering
	\includegraphics[width=.7\textwidth]{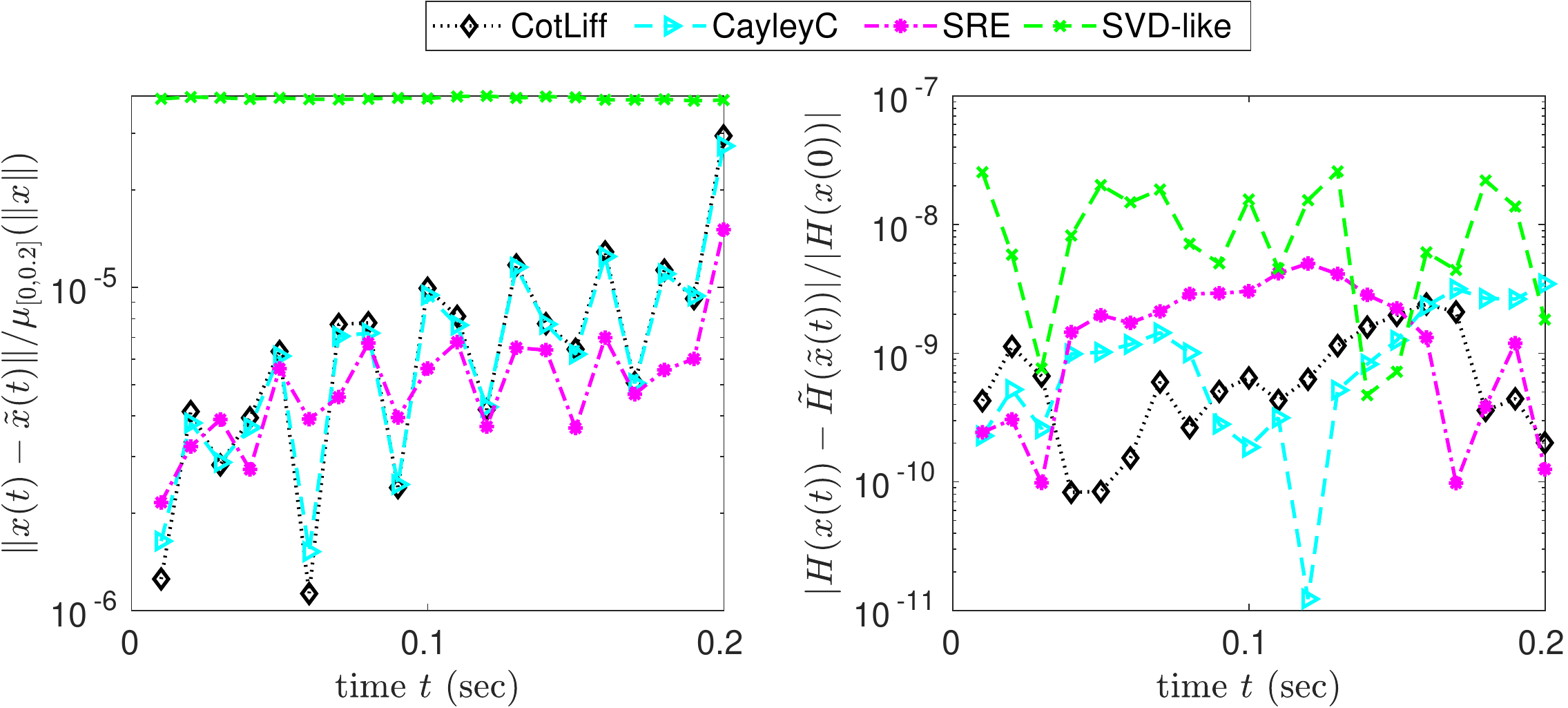}
	\caption{Vlasov equation, $n=1001$, $k = 10$: relative errors in the state vector (left) and the Hamiltonian function (right).
	}
	\label{fig:Vlasov}
\end{figure}

\subsubsection{Discussion}\label{subsec:mor_tests}

First, for different test models, the relative errors in the state ($\mathrm{RE}_x$)  generally decrease when the ROMs computed by using the optimization schemes are getting larger, see Figure~\ref{fig:error_vs_dim} for a~comparison.
However, the error in the energy  ($\mathrm{RE}_H$) sometimes does not behave the same. 
A~difference in the linear and nonlinear cases can be observed. Recalling that theoretically, if $x_0 \in \mbox{im}(U)$, then the error in the Hamiltonian function should be zero. However, in practice, due to different errors, most considerably caused by the time integration solver, the computed error in the Hamiltonian function is nonzero and apparently independent of the dimension of the ROMs, see, e.g.,  Table~\ref{tab:LinWaveEq} for the linear wave equation. 
Meanwhile, for nonlinear models, this error additionally suffers from the approximation of the nonlinear term.  In view of \cite[Theorem 5.1]{PengM16}, we can generally state that the more DEIM modes in approximating the nonlinear term we use, which is the number of columns of $V$ in \eqref{eq:SDEIM_scheme} and  \eqref{eq:Structure_Preserving_scheme}, the smaller error in the energy we obtain. In our implementation, this number is nailed to about $2.5 k$ and therefore explains the monotone decrease of the error versus reduced order in the case of nonlinear Hamiltonian systems.

\begin{figure}[!h]%
	\centering
	\includegraphics[width=0.65\textwidth]{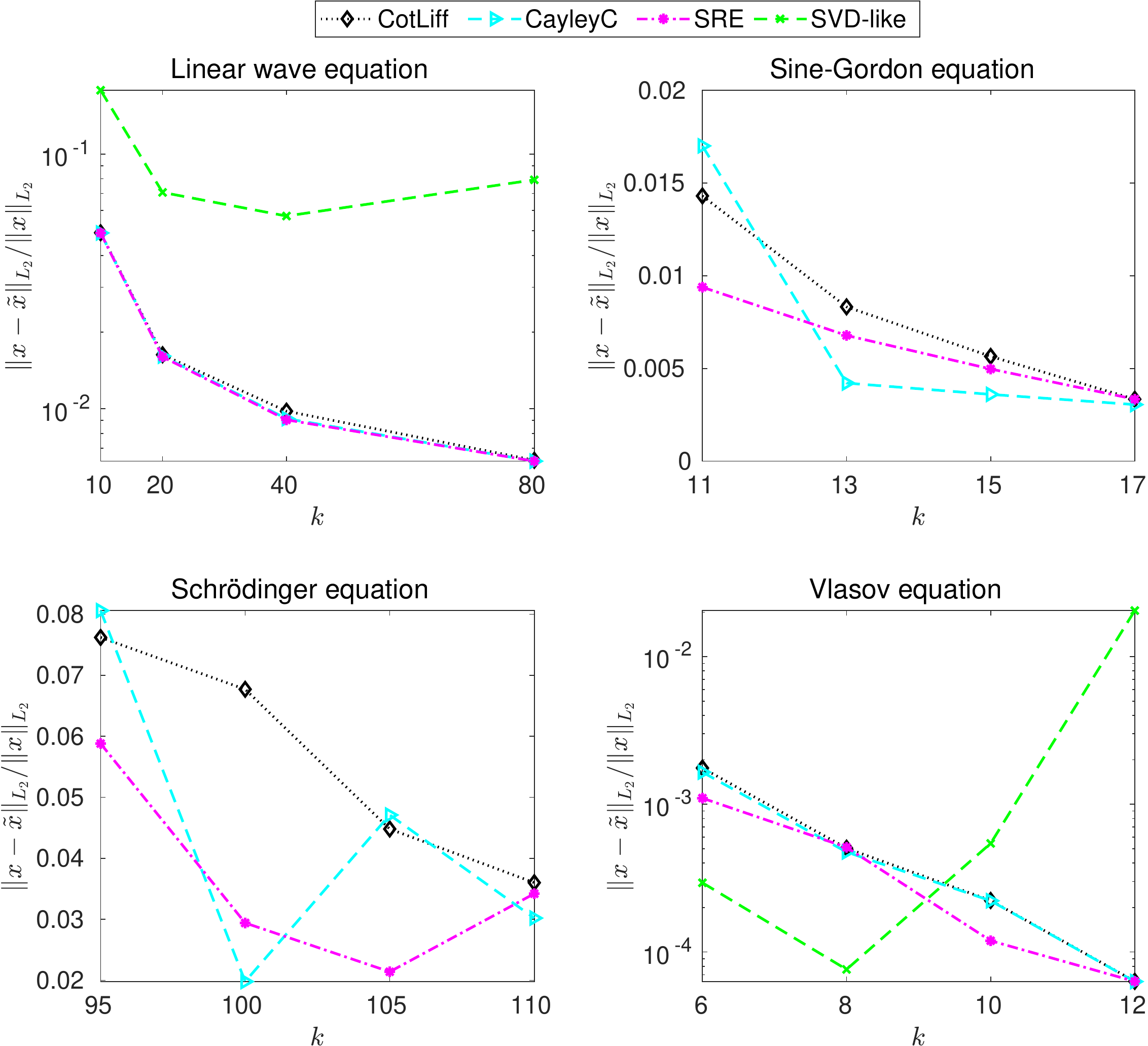}
	\caption{Relative errors $\mathrm{RE}_x$ in the state vector for the ROMs of different dimensions.}
	\label{fig:error_vs_dim}
\end{figure}

Second, as the cotangent lift method provides only a~suboptimal solution to the minimization problem~\eqref{eq:min_mor},
the optimization-based methods almost always improve the reduction results by delivering slightly smaller errors in both state vector and energy function.  
In the optimization aspect, this fact can be viewed as a result of the decrease in the value of the cost function~$f$. Figure~\ref{fig:valcost} reports the history of the values of the cost function for the wave and sine-Gordon equations. One can observe that $f(X_i)$ is decreasing in Algorithm~\ref{alg:non-monotone gradient} starting from a~suboptimal solution generated by the cotangent lift method, leading to a~reduction in errors. 
However, the correlation between the amounts of decrease in the values of the cost function 
and the resulting model reduction error $\mathrm{RE}_x$ is unclear. 
Furthermore, note that the amount of improvement is in general rather modest. On the one hand, the starting point is suboptimal which means it is already quite a good approximation. On the other hand, the optimization method 
here is of first order, i.e., only first-order information of the cost function is employed. This class of methods is known to be slow when approaching the limit. Moreover, the final model reduction error depends also on other factors, e.g., the approximation of the nonlinear term, which also attenuates the improvement. 

\begin{figure}[ht]%
	\centering
	\includegraphics[width=0.7\textwidth]{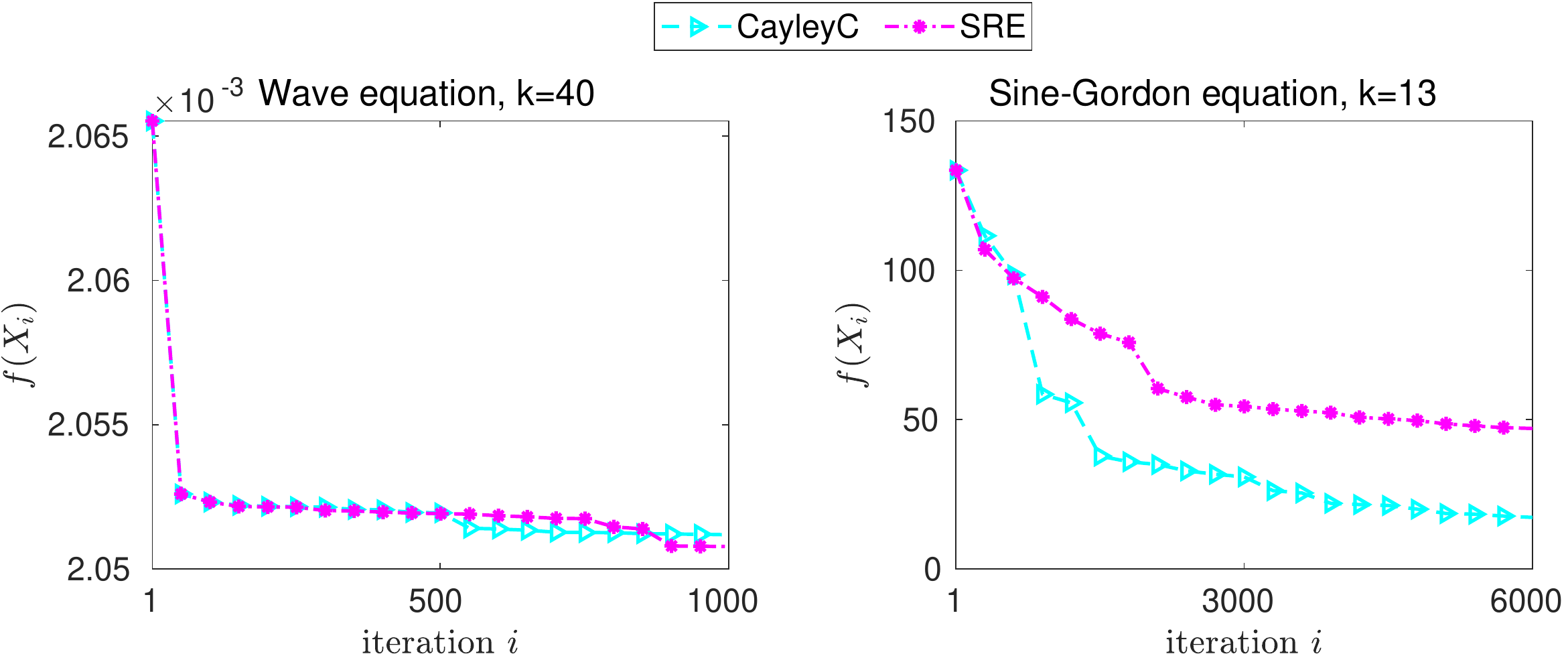}
	\caption{Values of the cost function for the wave and sine-Gordon equations provided by the CayleyC and SRE optimization schemes.}
	\label{fig:valcost}
\end{figure}

Third, unlike linear Hamiltonian systems, for which the simulation time  of ROMs 
linearly depends on the reduced order, that of nonlinear systems is  
difficult to predict as it also strongly depends on the structure of the nonlinear term. 
Moreover, if implicit integration schemes, like the Crank--Nicolson method in our case, are used for simulation, then the convergence properties of the employed nonlinear solver may significantly impact the computational time.
In view of this, we can explain the low accelerating factor for the Schr\"{o}dinger equation compared to other models. Indeed, the  nonlinear term $\big[(G(q,p)q)^T\; (G(q,p)p)^T\big]^T$ in the FOM \eqref{eq:SchroedingerHamiltonform2} allows a vectorization in computing which results in quite a~fast simulation of the FOM even in large dimension while that of the ROM does not have such a~structure any more. As a~consequence, even though the PSD-DEIM method is exploited, i.e., only a moderate number ($m$ instead of $2n$) of components of the nonlinear term  are evaluated, 
the simulation of the ROM is still quite time consuming, especially in designing a~selection strategy via a~selector matrix~$P$. Note also that for the Schr\"odinger model,  the ratio of the order reduction, $k/n$, is largest, which also contributes to the observed fact.

Fourth, the SVD-like approach does not work reliably for the sine-Gordon and Schr\"{o}dinger equations with the reported setting which makes us hesitate to present the result. Here, we used the available code from \cite{BuchBH19} and it is claimed there that the basis can deviate from being symplectic in some cases. This situation is unfortunately experienced in our experiments for these models, and also for the wave equation when $k = 40$.
It is most probably the reason why the simulation of the computed ROMs suffers from ill-conditioning and results in large errors. Nevertheless, the SVD-like reduction method delivers ROMs with the smallest errors in some situations for the Vlasov equation.
Further investigations are required to get a~full picture of this approach. 

Finally, presenting Table~\ref{tab:Vlasov_SDEIMvsStrucPresv}  with the relative errors in the state vector and energy for the reduced-order Vlasov models with different approximated nonlinear terms given in \eqref{eq:SDEIM_scheme} and \eqref{eq:Structure_Preserving_scheme}, we would like to convey the message that the structure preservation must be encompassed by	an~adequate approximation to have expectedly satisfactory result. The ROM with \eqref{eq:Structure_Preserving_scheme} is obviously Hamiltonian but a considerable error might have been caused during the approximation of the nonlinear term. As a consequence, the overall quality of the ROM obtained by the structure-preserving PSD-DEIM, including the energy preservation, is apparently not as good as that computed by the PSD-DEIM, which is not exactly structure-preserving.

\begin{table}[ht]
	\caption{Vlasov equation, $n=1001$, $k=6$: comparison of the PSD-DEIM and structure-preserving PSD-DEIM methods.}
	\label{tab:Vlasov_SDEIMvsStrucPresv}
	\centering
	\footnotesize
	\begin{tabular}{c r r r r r r r r}
		\toprule
		\multirow{2}{*}{PSD-DEIM} & \multicolumn{2}{c}{CotLift} & \multicolumn{2}{c}{CayleyC} & \multicolumn{2}{c}{SRE}&\multicolumn{2}{c}{SVD-like} \\\cmidrule(r){2-3}\cmidrule(r){4-5}\cmidrule(r){6-7}\cmidrule(r){8-9}
		& $\mathrm{RE}_x$ & $\mathrm{RE}_H$ & $\mathrm{RE}_x$ & $\mathrm{RE}_H$ & $\mathrm{RE}_x$ & $\mathrm{RE}_H$ & $\mathrm{RE}_x$ & $\mathrm{RE}_H$   \\\midrule
		\eqref{eq:SDEIM_scheme} & $1.76\e{-3}$ & $1.32\e{-6}$ &$1.65\e{-3}$ &$2.40\e{-7}$ &$1.10\e{-3}$ &$6.67\e{-7}$& $2.94\e{-4}$&$4.37\e{-9}$ \\
		\eqref{eq:Structure_Preserving_scheme} &$1.08\e{-2}$ &$6.35\e{-5}$ &$2.62\e{-3}$ &$6.34\e{-5}$ &$3.08\e{-3}$ & $6.35\e{-5}$ &$1.22\e{-2}$& $6.34\e{-5}$\\
		\bottomrule
	\end{tabular}
\end{table}

\section{Conclusion}\label{sec:concl}
We have proposed a new retraction on the symplectic Stiefel manifold which is based on an~SR decomposition. Its domain contains the unit ball and thus results in the global convergence of the corresponding Riemannian gradient-based optimization method. 

Various applications and examples have also been presented for validating and comparing the optimization schemes derived by combining different metrics and retractions. Numerical results showed that, depending on the problem and the setting, choice for metric and retraction must be taken with care to get better result. Especially,  schemes that use the SR~retraction maintains the symplecticity constraint the best. Running the same number of iteration, when $k$ is considerably smaller than $n$, there is not much difference in the time consumed by the schemes based on either the canonical-like metric or the Euclidean metric. However, if $k$ is approaching $n$, the schemes with canonical-like metric are faster. The numerical results also suggest that in this case, the Cayley retraction is favorably combined with the canonical-like metric as this combination tends to be faster than the others.

In addition, as an~accumulation point of the optimization iterates can be a~saddle point, deep investigation on the structure of  the set of critical points of the cost function is always helpful in assuring to compute a~minimizer. Finally, the fact that in many cases, especially in the symplectic model reduction problem, the presented optimization methods can only slightly improve the established model reduction methods, urges an investigation for faster optimization methods such as  conjugate gradient method and Newton method.

\section*{Acknowledgment}
\noindent
Part of this work was done when Bin Gao and Nguyen Thanh Son were with ICTEAM Institute, UCLouvain
and supported by the Fonds de la Recherche Scientifique -- FNRS and the Fonds Wetenschappelijk Onderzoek -- Vlaanderen under EOS Project no 30468160. The authors also would like to thank P.-A. Absil for helpful discussions.




\bibliographystyle{elsarticle-num}
\bibliography{references}


%
%
%
\end{document}